\documentclass[12pt]{scrartcl}
\usepackage[english]{babel}
\usepackage[utf8x]{inputenc}
\usepackage[T1]{fontenc}
\usepackage{lmodern}
\usepackage{amsfonts}
\usepackage{amsmath}
\usepackage{amssymb}
\usepackage{amsthm}
\usepackage{mathrsfs}
\usepackage[colorlinks=false,urlcolor=blue]{hyperref}
\usepackage{listings}
\usepackage{graphicx}
\usepackage[shortlabels]{enumitem}
\usepackage{mathtools}

\addtokomafont{section}{\rmfamily}
\addtokomafont{subsection}{\rmfamily}
\addtokomafont{subsubsection}{\rmfamily}
\addtokomafont{paragraph}{\rmfamily}
\usepackage{indentfirst}

\swapnumbers
\theoremstyle{plain}
\newtheorem{satz}[paragraph]{Theorem}
\newtheorem{lem}[paragraph]{Lemma}
\newtheorem{kor}[paragraph]{Corollary}
\newtheorem{prop}[paragraph]{Proposition}
\theoremstyle{definition}
\newtheorem{defn}[paragraph]{Definition}
\newtheorem{bem}[paragraph]{Remark}

\counterwithin{paragraph}{section}
\counterwithout{paragraph}{subsection}
\counterwithout{paragraph}{subsubsection}
\counterwithin{equation}{section}

\newcommand{\R}{\mathbb{R}}
\newcommand{\C}{\mathbb{C}}
\newcommand{\CP}{\mathbb{CP}}
\renewcommand{\O}{\mathrm{O}}
\newcommand{\Id}{\mathrm{Id}}
\newcommand{\Ric}{\mathrm{Ric}}
\newcommand{\scal}{\mathrm{scal}}
\newcommand{\vol}{\mathrm{vol}}
\newcommand{\Sym}{\operatorname{Sym}}
\newcommand{\tr}{\operatorname{tr}}
\newcommand{\im}{\operatorname{im}}
\newcommand{\diag}{\operatorname{diag}}
\newcommand{\End}{\operatorname{End}}
\newcommand{\Hom}{\operatorname{Hom}}
\newcommand{\Ad}{\operatorname{Ad}}
\newcommand{\ad}{\operatorname{ad}}
\newcommand{\X}{\mathfrak{X}}
\newcommand{\U}{\operatorname{U}}
\newcommand{\SU}{\operatorname{SU}}
\newcommand{\su}{\mathfrak{su}}
\renewcommand{\u}{\mathfrak{u}}
\newcommand{\SO}{\operatorname{SO}}
\newcommand{\so}{\mathfrak{so}}
\newcommand{\Sp}{\operatorname{Sp}}
\renewcommand{\sp}{\mathfrak{sp}}
\newcommand{\Spin}{\operatorname{Spin}}
\newcommand{\GL}{\operatorname{GL}}
\newcommand{\gl}{\mathfrak{gl}}
\makeatletter
\@ifundefined{sl}{%
}{
}
\makeatother
\newcommand{\spann}{\operatorname{span}}
\newcommand{\pr}{\operatorname{pr}}
\newcommand{\m}{\mathfrak{m}}
\newcommand{\h}{\mathfrak{h}}
\newcommand{\g}{\mathfrak{g}}
\newcommand{\p}{\mathfrak{p}}
\renewcommand{\k}{\mathfrak{k}}
\renewcommand{\v}{\mathfrak{v}}
\newcommand{\intprod}{\mathbin{\lrcorner}}
\newcommand{\rmE}{\mathrm{E}}

\newcommand{\rmG}{\mathrm{G}}
\newcommand{\fs}{\mathfrak{s}}
\newcommand{\Hol}{\mathrm{Hol}}
\newcommand{\hol}{\mathfrak{hol}}
\newcommand{\Stab}{\mathrm{Stab}}
\newcommand{\stab}{\mathfrak{stab}}
\newcommand{\rank}{\operatorname{rk}}
\newcommand{\Horiz}{\mathcal{H}}
\newcommand{\Verti}{\mathcal{V}}
\newcommand{\Curv}{\mathcal{K}}
\newcommand{\mixed}{\mathrm{m}}
\newcommand{\threead}{3-$(\alpha,\delta)$-Sasaki}
\newcommand{\threeadn}{3-$(\alpha,\delta)$-Sasakian}
\newcommand{\Lie}{\mathcal{L}}
\newcommand{\localmap}{\dashrightarrow}
\newcommand{\Cas}{\operatorname{Cas}}
\newcommand{\widesttilde}[1]{\big[#1\big]^\sim}

\title{\rmfamily Submersion constructions for geometries with parallel skew torsion}
\author{Andrei Moroianu\footnote{Université Paris-Saclay, CNRS,  Laboratoire de mathématiques d'Orsay, 91405 Orsay, France.}\ \footnote{Institute of Mathematics ``Simion Stoilow'' of the Romanian Academy, 21 Calea Grivitei, 010702 Bucharest, Romania.}, Paul Schwahn\footnote{Universidade Estadual de Campinas, IMECC, Rua Sérgio Buarque de Holanda 651, 13083-859 Campinas-SP, Brazil.}}
\date{}

\begin{document}

\maketitle

\begin{abstract}
\noindent
In the absence of a de Rham decomposition theorem for geometries with torsion, we develop and unify ways to view a geometry with parallel skew torsion as the total space of a locally defined, not necessarily unique Riemannian submersion with totally geodesic fibers. We complete and extend the Cleyton--Swann classification of irreducible such geometries and characterize the cases where the stabilizer of the torsion is larger than the holonomy. As a byproduct, we obtain structure results on Gray manifolds, nearly parallel $\rmG_2$-manifolds and Sasaki manifolds with reducible holonomy.

\medskip


\medskip

\end{abstract}

\section{Introduction}
\label{sec:intro}

The tangent bundle of every Riemannian manifold carries a unique torsion-free metric connection, the so-called Levi-Civita connection. This is the main tool in studying general Riemannian metrics. However, in the presence of additional geometric structure like nearly Kähler, Sasaki, 3-Sasaki, homogeneous, etc., it appears that specific metric connections with torsion preserving the given structure are better adapted in order to gather relevant information.

In their foundational paper \cite{CS}, Cleyton and Swann studied geometries with torsion, defined by a metric connection $\nabla^\tau$ whose torsion $\tau$ is totally skew-symmetric and parallel with respect to the connection itself, as is the case in all aforementioned special geometric contexts. Of course, one tacitly assumes that the torsion is non-vanishing, since otherwise the problem is empty. One remarkable result that they obtain is the following (except in dimension 3 which was overlooked): when the infinitesimal holonomy representation of the connection $\nabla^\tau$ is irreducible, then either the manifold is Ambrose--Singer (i.e. the curvature is parallel as well, and the manifold is, up to regularity, locally modelled on a naturally reductive homogeneous space), or the manifold is nearly Kähler in dimension 6, or nearly parallel $\rmG_2$ in dimension 7.

In the case of Riemannian (torsion-free) geometries, the de Rham decomposition theorem allows to understand a geometry with reducible holonomy representation (at least locally) as a product of irreducible factors corresponding to the summands of the holonomy representation. This is no longer the case in the presence of torsion, where even for a reducible holonomy representation, the torsion tensor may have ``mixed'' components which prevent the manifold from locally splitting as a product.

A systematic study of the general (holonomy-reducible) case was undertaken in \cite{CMS}. The main contribution there is the introduction of the so called {\em standard decomposition} of the tangent bundle of a geometry with parallel skew torsion, in horizontal and vertical summands. An irreducible component of the holonomy representation is called horizontal if there exists a non-trivial element of the holonomy group acting trivially on all other irreducible components, and vertical otherwise. The point is that with respect to this decomposition of the tangent bundle $TM=\Verti\oplus \Horiz$, the torsion form has vanishing projection onto $\Lambda^2\Verti\otimes \Horiz$, and this, in turn, ensures the existence of a local Riemannian submersion with totally geodesic fibres tangent to $\Verti$ from $M$ to some other local geometry with torsion of smaller dimension. Moreover the fibres are Ambrose--Singer manifolds, and the ``mixed'' part of the torsion coincides with one of the O'Neill tensors, which is the obstruction for the submersion to be of product form.

Using this construction, geometries with torsion were characterized in \cite{CMS} in terms of so-called geometries with parallel curvature, and in the particular case where the holonomy representation on $\Verti$ is trivial, a complete classification was obtained.

It turns out, however, that even on explicit geometries with torsion like Sasakian for instance, the holonomy group is not easy to describe in general. It is therefore useful sometimes to consider the decomposition of the tangent bundle as representation of other groups larger than the holonomy group, like the stabilizer of the torsion form $\tau$, or any other group inbetween.

The necessity of considering more general splittings into parallel subbundles than those used in the construction of the standard submersion was already pointed out in \cite{3adsubm} in the context of \threead\ manifolds. The crucial condition required for the construction of the submersion is that the torsion form has no $\Lambda^2\Verti\otimes\Horiz$-part, leading to our definition of \emph{admissible splittings} (Definition~\ref{admsplit}). We extend the notion of standard decomposition from \cite{CMS} to any intermediate algebra $\g$ with $\hol(\nabla^\tau)\subseteq\g\subseteq\stab(\tau)$ and show that most results from \cite{CS} and \cite{CMS} generalize to this setting. In particular, in Theorem \ref{csextended} we extend (and slightly correct and complete) Cleyton--Swann's result on geometries with torsion with irreducible holonomy representation to the case where the tangent bundle is irreducible as representation of $\stab(\tau)$ but possibly reducible with respect to $\hol(\nabla^\tau)$.

An important question here is to classify the possible holonomies of $\nabla^\tau$ when the stabilizer of the torsion is $\SU(3)$ or $\rmG_2$. The first group corresponds to so-called strict nearly Kähler manifolds in dimension six (also called Gray manifolds), and in the complete case, a classification was obtained in \cite{BM} and \cite{nagy}. We extend these results to the non-complete (local) setting in Proposition~\ref{nk6hol} and Theorem~\ref{nk6reducible}. The combined results read as follows:

\begin{satz}
 Let $(M^6,g,J)$ be a Gray manifold such that the holonomy algebra of the canonical connection is properly contained in $\su(3)$. Then either $(M,g,J)$ is locally isomorphic to the homogeneous Gray manifold $S^3\times S^3$ (see Rem.~\ref{homNK6}), or to the twistor space over an anti-self-dual Einstein $4$-manifold.
\end{satz}

Similarly, when the stabilizer of the torsion is $\rmG_2$, which corresponds to nearly parallel $\rmG_2$-structures, the classification of possible holonomies of $\nabla^\tau$ is obtained in Proposition~\ref{ng2hol} and Theorem~\ref{ng2reducible}. Again, we summarize the results:

\begin{satz}
 Let $(M^7,g,\varphi)$ be a nearly parallel $\rmG_2$-manifold such that the holonomy algebra of the canonical connection is properly contained in $\g_2$. The either $(M,g,\varphi)$ is locally isomorphic to the Berger space $\SO(5)/\SO(3)_{\mathrm{irr}}$ (see Rem.~\ref{homNG2}), or it is a \threead\  manifold with $\delta=5\alpha$, and $\varphi$ is its canonical $\rmG_2$-structure (see \cite{3ad}).
\end{satz}

In the final section, we consider the case where the stabilizer of the torsion acts almost irreducibly on the tangent bundle (in the sense that it has only two irreducible summands, one of which is 1-dimensional). In Theorem \ref{sasakisatz} we show that this characterizes Sasakian geometry, and in Theorems \ref{sasakirealred} and \ref{sasakirealirred} we classify the cases where the holonomy group of $\nabla^\tau$ is strictly contained in the stabilizer of $\tau$.

\section{Geometries with parallel skew torsion}

\subsection{Notation}

Before we begin, a few remarks on notation are in order. Let $(V,g)$ be a finite-dimensional Euclidean vector space. Throughout this article, we always identify $V\cong V^\ast$ using $g$; moreover, any $2$-form $\alpha\in\Lambda^2V\cong\so(V)$ corresponds to a skew-symmetric endomorphism via
\[g(\alpha(X),Y)=\alpha(X,Y),\qquad X,Y\in V.\]
In particular, $(X\wedge Y)Z=g(X,Z)Y-g(Y,Z)X$ for $X,Y,Z\in V$. Moreover, any $3$-form $\tau\in\Lambda^3V$ corresponds to a $(2,1)$-tensor by
\[g(\tau_XY,Z)=\tau(X,Y,Z),\qquad X,Y,Z\in V,\]
so that $X\intprod\tau=\tau_X\in\so(V)$.

Given an endomorphism $A\in\End V$, we denote with $A_\ast$ its action as a derivation on tensor powers of $V$. In particular, on exterior forms,
\begin{equation}
A_\ast\alpha=\sum_iAe_i\wedge(e_i\intprod\alpha),\qquad\alpha\in\Lambda^kV,\label{extaction}
\end{equation}
where $(e_i)$ is some orthonormal basis of $V$. For $2$-forms, we have the useful identity
\begin{equation}
 \alpha_\ast\beta=[\alpha,\beta],\qquad\alpha,\beta\in\Lambda^2V\cong\so(V).\label{lam2action}
\end{equation}

The notation thus introduced applies in particular to the tangent bundle $TM$ and the vector fields $\X(M)$ of any Riemannian manifold $(M,g)$. Let moreover $\nabla^g$ denote the Levi-Civita connection of $g$, and $R^g$ the Riemannian curvature tensor
\[R^g(X,Y)=[\nabla^g_X,\nabla^g_Y]-\nabla^g_{[X,Y]},\qquad X,Y\in\X(M).\]

\subsection{Skew torsion and decomposability}

\begin{defn}
A \emph{geometry with parallel skew torsion} $(M,g,\tau)$ is a Riemannian manifold $(M,g)$ together with a 3-form $\tau\in\Omega^3(M)$ which is parallel with respect to the connection $\nabla^\tau:=\nabla^g+\tau$. This connection is then a metric connection with parallel skew-symmetric torsion $T^\tau=2\tau$.
\end{defn}

The curvature $R^\tau$ of the connection $\nabla^\tau$ is related to the Riemannian curvature $R^g$ by
\begin{equation}
 R^g(X,Y)=R^\tau(X,Y)+[\tau_X,\tau_Y]-2\tau_{\tau_XY},\label{curvdiff}
\end{equation}
cf.~\cite[(2.2)]{CS}.

Throughout the article, we will repeatedly pick an arbitrary point of $M$ and isometrically identify the tangent space at that point with Euclidean $\R^n$. This allows us to view $\O(n)$ as the group of linear isometries of the tangent space and consider various subgroups thereof. There are two subgroups which play a distinguished role in the study of geometries with parallel skew torsion: the \emph{holonomy group} $\Hol(\nabla^\tau)$, which is contained in $\O(n)$ because $\nabla^\tau$ is metric, as well as the \emph{stabilizer} $\Stab(\tau)$ in $\O(n)$. Note that $\Hol(\nabla^\tau)\subset\Stab(\tau)$ because $\nabla^\tau\tau=0$.

Their corresponding Lie algebras are denoted by $\hol(\nabla^\tau)$ and $\stab(\tau)$, respectively, and subsequently abstractly viewed as subalgebras of $\so(n)$. Since all of our considerations in this article are of a local nature, we will mostly argue on the level of Lie algebras.

\begin{defn}
 A geometry with parallel skew torsion $(M,g,\tau)$ is called \emph{decomposable} if $TM=T_1\oplus T_2$ for nontrivial, orthogonal, $\nabla^\tau$-parallel distributions $T_1,T_2$, such that $\tau=\tau_1+\tau_2$ with $\tau_i\in\Lambda^3T_i$. Otherwise $(M,g,\tau)$ is called \emph{indecomposable}.
\end{defn}

\begin{bem}
By \cite[Lem.~3.2]{CMS}, a geometry with parallel skew torsion is decomposable if and only if it is locally isometric to a product of geometries with parallel skew torsion. In the torsion-free or Riemannian case ($\tau=0$), we recover the local de Rham decomposition theorem: decomposability is equivalent to reducibility of the holonomy representation.

However, a geometry with torsion may well have reducible holonomy representation while being indecomposable. The extent to which this equivalence fails for geometries with parallel skew torsion was investigated in \cite{CMS}, and we shall get back to precisely this issue in~\ref{sec:cansubm}. The discussion hinges critically on the torsion being parallel. All the more remarkable is the following result by Dileo--Lotta \cite{DL}: if a connected Riemannian manifold $(M,g)$ carries a metric connection with (not necessarily parallel) skew torsion whose holonomy representation is reducible, then $(M,g)$ is locally a product, provided the sectional curvature of $g$ is \emph{nonpositive}.
\end{bem}

\subsection{Some special cases}

In order to sharpen our intuition and at same time present some of the situations that will become important later on, let us review some well-known examples of geometries with parallel skew torsion.

\paragraph{Naturally reductive spaces.}
\label{natred}
 Let $(M=G/H,\m)$ be a \emph{reductive} homogeneous space, that is, there is an $\Ad(H)$-invariant splitting $\g=\h\oplus\m$. The \emph{reductive complement} $\m$ is canonically identified with the tangent space of $M$ at the identity coset. Such a manifold admits a \emph{canonical reductive connection} $\nabla$, which is the $H$-connection on the principal bundle $G\to G/H$ whose horizontal distribution in $TG$ consists of all left translates of $\m\subset\g$. It has the notable property that all invariant tensors on $M$ are $\nabla$-parallel -- in particular, its torsion $T$ and curvature $R$ are parallel, and given by
 \begin{align*}
  T(X,Y)&=-[X,Y]_\m,&R(X,Y)Z&=-[[X,Y]_\h,Z],&X,Y,Z\in\m,
 \end{align*}
 where the subscripts denote projections to the respective subspaces.

 If $g$ is an invariant Riemannian metric on $M$, then $(M,\m,g)$ is called a \emph{naturally reductive} homogeneous space if
 \[g([X,Y]_\m,Z)+g(Y,[X,Z]_\m)=0\qquad\forall X,Y,Z\in\m.\]
 Equivalently, the torsion of $\nabla$ is skew-symmetric with respect to $g$. By a famous theorem of Ambrose--Singer, any complete, simply connected geometry with parallel skew torsion $(M,g,\tau)$ and parallel curvature $R^\tau$ is a naturally reductive homogeneous space, and $\nabla^\tau$ coincides with the canonical connection.

 Motivated by this, we call a geometry with parallel skew torsion $(M,g,\tau)$ satisfying $\nabla^\tau R^\tau=0$ a \emph{naturally reductive Ambrose--Singer manifold}. Omitting the assumption of completeness introduces a subtlety: an Ambrose--Singer manifold is locally isometric to a homogeneous space if and only if it is \emph{regular} in the sense of Tricerri \cite{tricerri}, see also \cite{CCL}.

 Let $(M=G/H,\m,g)$ be a naturally reductive homogeneous space, and assume that $G/H$ is almost effective, that is,  the representation $\h\to\so(\m)\cong\Lambda^2\m$ is faithful. Since $\nabla$ has parallel curvature, the Ambrose--Singer holonomy theorem implies that
 \begin{equation}
 \hol(\nabla)=\im(R: \Lambda^2\m\to\Lambda^2\m)=[\m,\m]_\h\subseteq\h,
 \label{holhom}
 \end{equation}
 and moreoever this is an ideal in $\h$.

 In this spirit one may associate to any Ambrose--Singer manifold $(M,g,\tau)$ its \emph{transvection algebra} $\g:=\h\oplus\m$, where $\h:=\hol(\nabla^\tau)=\im R$, $\m$ is the tangent space at some point, and the missing part of the bracket on $\g$ is given by
 \begin{equation}
  [X,Y]:=(-R^\tau(X,Y),-T^\tau(X,Y))\in\g,\qquad X,Y\in\m.\label{transvec}
 \end{equation}
 Let $G$ be the simply connected Lie group with Lie algebra $\g$, and $H\subset G$ the connected subgroup with Lie algebra $\h$. By \cite[Prop.~3.1.21]{CCL}, if $H$ is closed in $G$, then $(M,g,\tau)$ is regular and locally isometric to the homogeneous space $G/H$.

 The naturally reductive spaces form a rich family, containing many interesting geometric examples. For a detailed description of their structure, see \cite{storm}. We would like to note here two particularly simple cases: first, \emph{isotropy irreducible} spaces, i.e.~homogeneous spaces $G/H$ whose isotropy representation $\m$ is irreducible, and the invariant metric is unique up to scale. These were classified by Wolf in 1968. If $G/H$ is not a symmetric space, then $G$ is necessarily compact and simple \cite[Thm.~1.1]{wolf}.

 Second, and as a special case of the first, \emph{symmetric spaces of group type} are the (resp.~compact and noncompact) symmetric spaces $(K\times K)/K$ and $K^\C/K$, where $K$ is a compact simple Lie group. These possess the notable property that the choice of reductive complement $\m$, and thus the canonical connection, is not unique. We illustrate this for the space $G/H=(K\times K)/K$, where
 \begin{align*}
  \g&=\k\oplus\k,&\h&=\diag(\k)=\{(X,X)\,|\,X\in\k\},\\
  &&\m_t&=\{((t-1)X,(t+1)X)\,|\,X\in\k\}
 \end{align*}
 defines a one-parameter family of reductive decompositions $\g=\h\oplus\m_t$. The associated family of canonical connections $\nabla^t$ has torsion and curvature given by
 \begin{align*}
  T^t(X,Y)&=-2t[X,Y],&R^t(X,Y)Z&=-(1-t^2)[[X,Y],Z],&X,Y,Z\in\k,
 \end{align*}
 under the identification $\k\stackrel{\sim}{\to}\m: X\mapsto((t-1)X,(t+1)X)$. In particular, $\nabla^0$ is the Levi-Civita connection of any bi-invariant metric on $K$, while $\nabla^{\pm1}$ are flat connections also called the \emph{Cartan $(\pm)$-connections} on $K$.

\paragraph{Gray manifolds.}
 \label{nearlykaehler}
 A \emph{nearly Kähler} manifold $(M,g,J)$ is an almost Hermitian manifold satisfying the condition
 \[(\nabla^g_XJ)X=0\qquad\forall X\in TM.\]
 Any almost Hermitian manifold possesses a \emph{canonical Hermitian connection} $\nabla$ such that $g$ and $J$ are $\nabla$-parallel. It is given by
 \[\nabla=\nabla^g+\tau,\qquad \tau_XY=-\frac{1}{2}J(\nabla^g_XJ)Y,\]
 and its torsion is skew-symmetric if and only if $(M,g,J)$ is nearly Kähler. This also implies that the torsion of $\nabla$ is parallel, i.e.~$(M,g,\tau)$ is a geometry with parallel skew torsion. Moreover $\tau$ is of type $(3,0)+(0,3)$ with respect to $J$, that is
 \[\tau_X\circ J+J\circ\tau_X=0,\qquad X\in TM.\]

 A structure result by Nagy states that any strict (i.e.~non-Kähler) nearly Kähler manifold locally splits as a product of factors which may be homogeneous, twistor spaces over quaternion-Kähler manifolds, or $6$-dimensional \cite{nagy}. Strict nearly Kähler $6$-manifolds thus have a special status; they are also called \emph{Gray manifolds}. Identifying the tangent space with $\R^6$, the stabilizer of the torsion in $\so(6)$ is given by $\stab(\tau)=\su(3)$ for Gray manifolds, since in this case $\tau$ is the real part of a complex volume form. Conversely, we have the following:

\begin{lem}
\label{concludeNK6}
 Any geometry with parallel skew torsion $(M^6,g,\tau)$ with $\stab(\tau)=\su(3)$ is a Gray manifold with respect to some almost complex structure $J$.
\end{lem}
\begin{proof}
Up to scaling, $\stab(\tau)=\su(3)$ stabilizes exactly one nonzero vector $J\in\Lambda^2\R^6\cong\so(6)$. The same goes for $\Sym^2\R^6$, so $J^2$ is a multiple of the identity, and we can assume that $J^2=-\Id$. Since $\hol(\nabla^\tau)\subseteq\stab(\tau)$, it follows that there exists an almost Hermitian structure $J$ on $(M,g)$ which is $\nabla^\tau$-parallel. Moreover, $\tau$ is of type $(3,0)+(0,3)$ with respect to $J$, because $(\Lambda^3\R^6)^{\su(3)}=\Lambda^{(3,0)+(0,3)}\mathbb{R}^6.$ This shows that $\tau_XJX=0$ for every tangent vector $X$, whence
$$(\nabla_XJ)(X)=(\nabla^\tau_XJ-\tau_XJ)(X)=-(\tau_XJ)(X)=-\tau_XJX+J(\tau_XX)=0.$$
Moreover since $\tau$ is nonzero, so is $\nabla J$, and thus $(M,g,J)$ is a Gray manifold.
\end{proof}

\begin{bem}
\label{homNK6}
 Homogeneous Gray manifolds have been classified by Butruille \cite{butruille}. Any simply connected homogeneous Gray manifold $G/H$ is one of the naturally reductive spaces
 \[S^6=\frac{\rmG_2}{\SU(3)},\qquad F_{1,2}=\frac{\SU(3)}{T^2},\qquad\CP^3=\frac{\SO(5)}{\U(2)},\qquad S^3\times S^3=\frac{S^3\times S^3\times S^3}{\diag(S^3)},\]
 each (up to scale) with the metric induced by minus the Killing form of $G$. The canonical Hermitian connection coincides with the canonical reductive connection, and with the help of \eqref{holhom} one may verify that in each of the above cases, $\hol(\nabla)=\h$.
\end{bem}

\paragraph{Nearly parallel $\rmG_2$-manifolds.}
 \label{nearlyg2}
 A \emph{$\rmG_2$-structure} on a $7$-dimensional manifold $M$ is a 3-form $\varphi\in\Omega^3(M)$ such that at every point, its stabilizer in $\GL(7,\R)$ is isomorphic to the compact Lie group $\rmG_2$. Any $\rmG_2$-structure $\varphi$ determines both a metric $g_\varphi$ and an orientation on $M$. The triple $(M,g_\varphi,\varphi)$ is called a \emph{nearly parallel} $\rmG_2$-manifold if
 \[d\varphi=\tau_0\star_\varphi\varphi\]
 for some $\tau_0\in\R$. Any nearly parallel $\rmG_2$-manifold carries a \emph{canonical $\rmG_2$-connection} $\nabla$ which parallelizes $\varphi$ and which is given by
 \[\nabla=\nabla^{g_\varphi}+\tau,\qquad\tau=\frac{\tau_0}{12}\varphi,\]
 see e.g.~\cite{AS}.\footnote{In contrast to \cite{AS}, there is a sign change in both $\tau_0$ and the choice of orientation induced by $\varphi$.} By definition, $\stab(\tau)=\g_2$. If $g$ is any other Riemannian metric on $M$ such that $\so(7,g)$ contains $\g_2$, then by the irreducibility of $\R^7$ under $\g_2$ and Schur's Lemma, $g$ must be a constant multiple of $g_\varphi$. We still call $(M,g,\varphi)$ a nearly parallel $\rmG_2$-manifold.

 Not all geometries with parallel skew torsion such that $\hol(\nabla^\tau)\subseteq\g_2$ are nearly parallel $\rmG_2$, because the torsion is not a priori linearly related to the $\rmG_2$-structure. Friedrich \cite{g2paralleltorsion} achieved a classification under the assumption that the $\rmG_2$-structure is \emph{cocalibrated} (meaning that $d^\ast\varphi=0$) and that the holonomy algebra is nonabelian.

\begin{lem}
\label{concludeNG2}
 Any geometry with parallel skew torsion $(M^7,g,\tau)$ such that $\stab(\tau)=\g_2$ is a nearly parallel $\rmG_2$-manifold.
\end{lem}
\begin{proof}
 Since $\g_2$ stabilizes exactly one vector in $\Lambda^3\mathbb{R}^7$ up to scaling, $\tau$ is proportional to a  $\rmG_2$-form $\varphi$. One can write $\tau=\frac{\tau_0}{12}\varphi$ for some non-zero constant $\tau_0\in\mathbb{R}$. Moreover, it is easy to check that $\sum\varphi_{e_i}\wedge\varphi_{e_i}=6\star_\varphi\varphi$ for every local orthonormal frame $(e_i)$. Using this, and the fact that $\nabla^\tau\tau=0$, we compute
 \begin{align*}
  d\varphi&=\frac{12}{\tau_0}d\tau=\frac{12}{\tau_0}\sum e_i\wedge\nabla_{e_i}\tau=-\frac{12}{\tau_0}\sum e_i\wedge(\tau_{e_i})_*\tau=-\frac{\tau_0}{12}\sum e_i\wedge(\varphi_{e_i})_*\varphi\\
  &=-\frac{\tau_0}{12}\sum e_i\wedge\varphi_{e_i}e_j\wedge\varphi_{e_j}=\frac{\tau_0}{6}\sum \varphi_{e_j}\wedge\varphi_{e_j}=\tau_0\star_\varphi\varphi.
 \end{align*}
\end{proof}

\begin{bem}
\label{homNG2}
 The compact, simply connected, homogeneous nearly parallel $\rmG_2$-mani\-folds have been classified by Friedrich et al.~\cite{nearlyg2}. A notable example is the \emph{Berger space} $\SO(5)/\SO(3)_{\mathrm{irr}}$, where the inclusion $\SO(3)_{\mathrm{irr}}\hookrightarrow\SO(5)$ is given by the $5$-dimensional irreducible representation of $\SO(3)$. This is an isotropy irreducible space where the canonical $\rmG_2$-connection coincides with the canonical reductive connection. Again one may check using \eqref{holhom} that $\hol(\nabla)=\so(3)$, viewed as a subalgebra of $\so(7)$ by its $7$-dimensional irreducible representation.
\end{bem}

\paragraph{Sasaki, 3-Sasaki, and 3-$(\alpha,\delta)$-Sasaki manifolds.}
\label{sasakibsp}
 For our purposes, a \emph{Sasaki manifold} $(M^{2n+1},g,\xi,\Phi)$ is a Riemannian manifold $(M,g)$ together with a unit length Killing vector field $\xi\in\X(M)$ and a skew-symmetric endomorphism field $\Phi$ satisfying
 \begin{equation}
 d\xi=2\Phi,\qquad \Phi^2=-\Id+\xi\otimes\xi,\qquad \nabla^g_X\Phi=-X\wedge\xi\quad\forall X\in TM.\label{sasaki}
 \end{equation}
 The \emph{canonical Sasaki connection} $\nabla$ is a metric connection given by
 \[\nabla=\nabla^g+\tau,\qquad\tau=\xi\wedge\Phi,\]
 and it has parallel skew-symmetric torsion.

 The endomorphism $\Phi$ annihilates $\xi$ and restricts to a complex structure on the orthogonal complement $\xi^\perp$. Thus the stabilizer of $\tau$ is isomorphic to $\U(n)$ if $\dim M=2n+1$. It is well known that any Sasaki manifold locally fibers over a Kähler manifold with fibers tangent to $\xi$.

 A \emph{3-Sasaki manifold} is a Riemannian manifold $(M^{4n+3},g)$ carrying three Sasakian structures $(\xi_i,\Phi_i)$, $i=1,2,3$, interacting via
 \begin{align}
  \xi_k&=-\Phi_i\xi_j=\Phi_j\xi_i,
  \label{3sas1}\\
  \Phi_kX&=-\Phi_i\Phi_jX+\langle\xi_j,X\rangle\xi_i=\Phi_j\Phi_iX-\langle\xi_i,X\rangle\xi_j\qquad\forall X\in TM,
  \label{3sas2}
 \end{align}
 for any even permutation $(i,j,k)$ of $(1,2,3)$.\footnote{Note the change in sign compared to \cite[Def.~1.2.2]{3ad} due to a different way of associating $2$-forms with endomorphisms.} In particular, the vector fields $\xi_i$ satisfy the $\so(3)$-commutation relations $[\xi_i,\xi_j]=2\xi_k$.
 
 For later use, we introduce the vertical distribution $\Verti$ spanned by $\xi_1,\xi_2,\xi_3$, its orthogonal complement $\Horiz:=\Verti^\perp$, and the horizontal endomorphisms $\Phi_i^\Horiz:=\Phi_i+\xi_j\wedge\xi_k$, which by \eqref{3sas1}-\eqref{3sas2} satisfy $\Phi_i^\Horiz(\xi_j)=0$ for every $i,j$ and 
 \begin{align}
  \Phi_i^\Horiz \Phi_j^\Horiz=- \Phi_j^\Horiz \Phi_i^\Horiz=-\Phi_k^\Horiz,
  \label{3sas2h}
 \end{align}
for every even permutation $(i,j,k)$ of $(1,2,3)$.

 A generalization of the latter are \emph{\threead\ manifolds}, introduced in \cite{3ad}. Again, these are Riemannian manifolds $(M^{4n+3},g)$ with three unit length Killing vector fields $\xi_i$ and skew-symmetric endomorphism fields $\Phi_i$ satisfying \eqref{3sas1} and \eqref{3sas2}, but with the condition \eqref{sasaki} replaced by
 \begin{align}
  \Phi_i^2&=-\Id+\xi_i\otimes\xi_i,\label{3adcond1}\\
  d\xi_i&=2\alpha\Phi_i+2(\alpha-\delta)\xi_j\wedge\xi_k\label{3adcond2}
 \end{align}
 for some constants $\alpha,\delta\in\R$, $\alpha\neq0$. The analogous first order condition on the $\Phi_i$ is then automatic \cite[Prop.~2.3.2]{3ad}. A \threead\ manifold can be obtained by rescaling a 3-Sasakian structure with two separate parameters on the distribution $\Verti$ spanned by $\xi_1,\xi_2,\xi_3$ and on its orthogonal complement $\Horiz:=\Verti^\perp$. Indeed, any \threead\ manifold with $\alpha\delta>0$ is related to a 3-Sasaki manifold in this way. Moreoever, any \threead\ manifold fibers over a quaternion-Kähler orbifold \cite{3adsubm,stecker}, and the distribution $\Verti$ is tangent to the fibers. For $\delta=0$ (the \emph{degenerate} case), the base of the fibration is hyper-Kähler.

 On any \threead\ manifold, there exists a family of metric connections $(\nabla^\gamma)_{\gamma\in\R}$ with skew-symmetric torsion, the so-called \emph{compatible connections}
 \[\nabla^\gamma=\nabla^g+\tau^\gamma,\qquad\tau^\gamma=\frac{\alpha}{2}\sum_{i=1}^3\xi_i\wedge\Phi_i^\Horiz+\frac{\gamma}{2}\xi_1\wedge\xi_2\wedge\xi_3,\]
 which have the distinctive property that they parallelize the distributions $\Verti$ and $\Horiz$. Only for the choice $\gamma=2(\delta-4\alpha)$ is the torsion also parallel; this is called the \emph{canonical \threead\ connection}.

 We remark that only in the so-called \emph{parallel} case $\delta=2\alpha$ does the canonical \threead\ connection parallelize the $\xi_i$ and $\Phi_i$. In this case, for any $\R$-linear combination $\xi$ of $\xi_1,\xi_2,\xi_3$, there exists a local fibration along $\xi$ over a nearly Kähler manifold whose canonical connection has reducible holonomy, and which in turn is locally isometric to the twistor space over the aforementioned quaternion-Kähler manifold \cite{steckerNK}.

 On a seven-dimensional \threead\ manifold ($n=1$), the three-form
 \[\varphi=\sum_{i=1}^3\xi_i\wedge\Phi_i^\Horiz+\xi_1\wedge\xi_2\wedge\xi_3\]
 gives a cocalibrated $\rmG_2$-structure \cite[Thm.~4.5.1]{3ad}, which is nearly parallel (and thus a multiple of the torsion of the canonical connection) if and only if $\delta=5\alpha$.

\begin{lem}
 For $n\neq1$, the stabilizer in $\SO(4n+3)$ of any of the $3$-forms $\tau^\gamma$ is isomorphic to $\Sp(n)\cdot\Sp(1)$, acting on $\Horiz\cong\R^{4n}$ in the standard way, and acting on $\Verti\cong\R^3$ by the adjoint representation of the factor $\Sp(1)$.
\end{lem}
\begin{proof}
It is enough to prove this at the infinitesimal level. Let $A\in \so(4n+3)$ be any element, decomposed with respect to the splitting $\R^{4n+3}=\Horiz\oplus\Verti$ as
$$A=a_1\xi_2\wedge\xi_3+a_2\xi_3\wedge\xi_1+a_3\xi_1\wedge\xi_2+\sum_{i=1}^3\xi_i \wedge X_i+\sigma,$$
where $X_i\in \Horiz$ and $\sigma\in\Lambda^2\Horiz$. An easy computation yields
\begin{align*}
 A_\ast(2\tau^\gamma)&=a_1(\xi_3\wedge\Phi_2^\Horiz-\xi_2\wedge\Phi_3^\Horiz)+a_2(\xi_1\wedge\Phi_3^\Horiz-\xi_3\wedge\Phi_1^\Horiz)+a_3(\xi_2\wedge\Phi_1^\Horiz-\xi_1\wedge\Phi_2^\Horiz)\\
&\quad+\sum_{i=1}^3X_i\wedge\Phi_i^\Horiz-\sum_{i,j=1}^3\xi_i\wedge\xi_j\wedge\Phi_i^\Horiz(X_j)\\
&\quad+\gamma(X_1\wedge\xi_2\wedge\xi_3+\xi_1\wedge X_2\wedge\xi_3+\xi_1\wedge\xi_2\wedge X_3)+\sum_{i=1}^3\xi_i \wedge \sigma_\ast(\Phi_i^\Horiz).
\end{align*}
Comparing types, we see that $A\in\stab(\tau^\gamma)$ if and only if the following system holds:
\[\begin{cases} \sigma_*(\Phi_i^\Horiz)=a_k\Phi_j^\Horiz-a_j\Phi_k^\Horiz\\
\sum_{i=1}^3X_i\wedge\Phi_i^\Horiz=0\\
\Phi_i^\Horiz(X_j)-\Phi_j^\Horiz(X_i)=\gamma X_k
\end{cases}\]
for every even permutation $(i,j,k)$ of $(1,2,3)$. Using \eqref{3sas2h}, the first condition is equivalent to $ (\sigma_0)_*(\Phi_i^\Horiz)=0$ for $i=1,2,3$, where $\sigma_0:=\sigma-\frac12\sum_{i=1}^3a_i \Phi_i^\Horiz$. Of course, the relations $ (\sigma_0)_*(\Phi_i^\Horiz)=0$ are equivalent to $\sigma_0\in\sp(n)$.

It remains to show that $X_i=0$ for $i=1,2,3$, since then it will follow that every element $A\in\stab(\tau^\gamma)$ can be written as a sum of an element in $\sp(n)$ and one in $\sp(1)$, where the elements $e_i$ of the standard basis of $\sp(1)$ act as $-\Phi_i^\Horiz$ on $\Horiz$ and as $2\xi_j\wedge\xi_k$ on $\Verti$ for every even permutation $(i,j,k)$ of $(1,2,3)$.

In order to prove our claim, we contract the second equation of the system with $\Phi_i^\Horiz$, and apply $\Phi_i^\Horiz$ to the resulting equation. Using \eqref{3sas2h} again, we obtain $$0=(2n-1)\Phi_i^\Horiz(X_i)+\Phi_j^\Horiz(X_j)+\Phi_k^\Horiz(X_k)$$ for every even permutation $(i,j,k)$ of $(1,2,3)$. This system immediately implies that if $2n-1\ne 1$, then $\Phi_i^\Horiz(X_i)=0$ for every $i$, so $X_i=0$.
\end{proof}

\clearpage

\section{Canonical splittings and submersions}
\label{sec:cansubm}

We reminded ourselves in \ref{sasakibsp} that some geometries with parallel skew torsion canonically bring locally defined submersions with them. This phenomenon has already been systematically studied in \cite{CMS} by means of the so-called \emph{standard submersion}. We recall and at the same time generalize some of the results of \cite{CMS} below.

\begin{defn}
\label{admsplit}
 Let $(M,g,\tau)$ be a geometry with parallel skew torsion. An \emph{admissible splitting} is an orthogonal decomposition of the tangent bundle into $\nabla^\tau$-parallel distributions $TM=\Horiz\oplus\Verti$ such that the projection of $\tau$ onto $\Lambda^2\Verti\otimes\Horiz\subset\Lambda^3TM$ vanishes.

 $\Horiz$ is called the \emph{horizontal distribution} and $\Verti$ the \emph{vertical distribution} of the admissible splitting.
\end{defn}

Given an admissible splitting, we may decompose the 3-form $\tau$ into a purely horizontal, a mixed, and a purely vertical part:
\[\tau=\tau^\Horiz+\tau^\mixed+\tau^\Verti,\]
where $\tau^\Horiz\in\Lambda^3\Horiz$, $\tau^\mixed\in\Lambda^2\Horiz\otimes\Verti$, and $\tau^\Verti\in\Lambda^3\Verti$. Clearly, if a geometry with parallel skew torsion is decomposable (and thus locally a product) it has an admissible splitting such that $\tau^\mixed=0$.

\paragraph{Local submersions.}\label{localsubm}
As remarked in \cite[Rem.~3.15]{CMS}, see also \cite[Thm.~2.1]{steckerNK}, any admissible splitting locally defines a Riemannian submersion $\pi: (M,g)\localmap(N,g_N)$ which enjoys the following properties:
\begin{enumerate}[(S1)]
 \item $N$ is the local leaf space of the integrable distribution $\Verti$, that is, $\Verti$ is actually the vertical distribution of $\pi$.
 \label{vertdist}
 \item $\pi$ has totally geodesic fibers.
 \item If the horizontal distribution admits a further $\nabla^\tau$-parallel splitting $\Horiz=\bigoplus_\alpha\Horiz_\alpha$, then the horizontal and mixed parts of $\tau$ decompose accordingly:
 \[\tau^\Horiz\in\bigoplus_\alpha\Lambda^3\Horiz_\alpha,\qquad\tau^\mixed\in\bigoplus_\alpha\Lambda^2\Horiz_\alpha\otimes\Verti.\]
 \label{tausplit}
 \item For all $V\in\Verti$, one has $(\tau_V)_\ast\tau^\Horiz=0$.
 \item The horizontal part $\tau^\Horiz$ is projectable to $N$, i.e.~$\tau^\Horiz=\pi^\ast\sigma$ for some $\sigma\in\Omega^3(N)$, and $(N,g_N,\sigma)$ is again a geometry with parallel skew torsion. The connections $\nabla^\tau$ and $\nabla^\sigma$ are $\pi$-related.
 \label{base}
 \item By restriction, every fiber $F$ of $\pi$ becomes a geometry with parallel skew torsion $(F,g_F,\tau^\Verti)$.
 \label{fiber}
 \item For any $X,Y\in\Horiz$ and $V\in\Verti$, we have the curvature identities
 \begin{align}
  R^\tau(X,V)&=0,\label{curvHV}\\
  R^\tau(X,Y)V&=-4[\tau_X,\tau_Y]V+4\tau_{\tau_XY}V.\label{curvHHV}
 \end{align}
 \label{curvature}
\end{enumerate}

In fact, the O'Neill invariant $A$ which measures the failure of the horizontal distribution to be integrable (see \cite[\S9.C]{besse}) is encoded in the mixed part of the torsion, $A=-\tau^\mixed$.

 The key observation leading to the next definition is the following: since $\nabla^\tau\tau=0$, the holonomy algebra $\hol(\nabla^\tau)$ is contained in the stabilizer algebra $\stab(\tau)$ of $\tau$ inside $\so(n)$. However holonomy and stabilizer do not necessarily coincide, and indeed fail to do so in many geometrically interesting cases. This leads us to the following scheme where ``parallel distributions'' are replaced with ``$\g$-invariant subbundles''.

\begin{defn}
 Let $\g\subseteq\so(n)$ be a Lie algebra such that $\hol(\nabla^\tau)\subseteq\g\subseteq\stab(\tau)$. The representation of $\g$ on $\R^n$ decomposes into an orthogonal sum of irreducible modules $\h_\alpha$ and $\v_j$ such that $\so(\h_\alpha)\cap\g\neq0$ for all $\alpha$ and $\so(\v_j)\cap\g=0$ for all $j$. Then the \emph{canonical $\g$-splitting} of $TM$ is defined by taking $\Horiz$ to be the subbundle associated to $\h=\bigoplus_\alpha\h_\alpha$, and $\Verti$ the subbundle associated to $\v=\bigoplus_j\v_j$.
\end{defn}

Even though the decomposition into irreducible $\g$-modules might itself not be unique, the canonical $\g$-splitting is, since any two isotypical summands belong to $\Verti$ by definition.

\begin{lem}
 Any canonical $\g$-splitting is an admissible splitting.
\end{lem}
\begin{proof}
 We give an adaptation of the proofs of \cite[Lem.~3.4, Lem.~3.6]{CMS}. Let for any $\alpha$ be $\h_\alpha^\perp$ the orthogonal complement of $\h_\alpha$ in $\R^n$, and denote $\g_\alpha:=\so(\h_\alpha)\cap\g$.

 First, we show that the representation of $\g_\alpha$ on $\h_\alpha$ has no trivial subspace, i.e.~$\h_\alpha^{\g_\alpha}=0$. Since $\g$ preserves the splitting $\R^n=\h_\alpha\oplus\h_\alpha^\perp$, we have $[\g,\g_\alpha]\subset\g_\alpha$, i.e.~$\g_\alpha$ is an ideal of $\g$. Thus, for every $v\in\h_\alpha^{\g_\alpha}$, $A\in\g_\alpha$ and $B\in\g$,
 \[ABv=[A,B]v+BAv=0,\]
 which shows that $\h_\alpha^{\g_\alpha}$ is a $\g$-invariant subspace of $\h_\alpha$. By assumption, $\h_\alpha$ is an irreducible $\g$-module, so $\h_\alpha^{\g_\alpha}=0$ or $\h_\alpha^{\g_\alpha}=\h_\alpha$. By definition, $\g_\alpha$ acts faithfully on $\h_\alpha$, and by assumption, $\g_\alpha\neq0$. Thus $\g_\alpha$ cannot act trivially on $\h_\alpha$ and we conclude that $\h_\alpha^{\g_\alpha}=0$.

 Next, since $\g_\alpha$ acts trivially on $\h_\alpha^\perp$ by definition, we have
 \[(\h_\alpha\otimes\Lambda^2\h_\alpha^\perp)^{\g_\alpha}=\h_\alpha^{\g_\alpha}\otimes\Lambda^2\h_\alpha^\perp\]
 and since $\h_\alpha^{\g_\alpha}=0$, we conclude that $(\h_\alpha\otimes\Lambda^2\h_\alpha^\perp)^{\g_\alpha}=0$.

 In particular $(\h_\alpha\otimes\Lambda^2\h_\alpha^\perp)^{\g}=0$, and it follows that
 \[(\Lambda^3\R^n)^\g\subset\left(\bigoplus_\alpha\Lambda^3\h_\alpha\right)\oplus\left(\bigoplus_\alpha\Lambda^2\h_\alpha\otimes\v\right)\oplus\Lambda^3\v.\]
 Since $\hol(\nabla^\tau)\subset\g$, the splitting $TM=\Horiz\oplus\Verti$ associated to $\R^n=\h\oplus\v$ is $\nabla^\tau$-parallel. And since $\g\subset\stab(\tau)$, the 3-form $\tau$ corresponds to an element of $(\Lambda^3\R^n)^\g$, whence the $\Lambda^2\Verti\otimes\Horiz$-part of $\tau$ vanishes.
\end{proof}

As a consequence of \ref{localsubm}, a canonical $\g$-splitting locally defines a Riemannian submersion called the \emph{canonical $\g$-submersion}. The \emph{standard submersion} defined in \cite{CMS} is just one example of a canonical $\g$-submersion, namely the case where $\g=\hol$.

\begin{bem}
 It is natural to ask whether every admissible splitting coincides with a canonical $\g$-splitting for a suitable choice of $\g$. The answer is no: in \cite{steckerNK} it is shown that fixing a  Reeb vector field $\xi$ on a parallel \threead\ manifold $(M,g,\tau)$ defines an admissible splitting with $\Verti=\R\xi$, which locally gives rise to a submersion over a nearly Kähler manifold whose torsion does not vanish. As a consequence of \ref{base} and Lemma~\ref{specialtype}, this is never a canonical $\g$-splitting if $(M,g,\tau)$ is indecomposable.
\end{bem}

Canonical $\g$-splittings have more distinguishing properties making them preferable. The following generalizations of \cite[Prop.~3.13]{CMS} and \cite[Lem~6.2]{CMS} are proved in complete analogy, while noting that $\g$-invariance implies parallelity under $\nabla^\tau$ since $\hol(\nabla^\tau)\subseteq\g$.

\begin{lem}
\label{verthom}
 Let $TM=\Horiz\oplus\Verti$ be a canonical $\g$-splitting. Then the composition
 \[\pr_{\Lambda^2\Verti}\circ \,R^\tau:\quad\Lambda^2TM\longrightarrow\Lambda^2\Horiz\oplus\Lambda^2\Verti\longrightarrow\Lambda^2\Verti\]
 is $\g$-invariant.

 In particular, any fiber $F$ of the canonical $\g$-submersion has parallel skew torsion and parallel curvature, so $F$ is a naturally reductive Ambrose--Singer manifold.
\end{lem}

\begin{lem}
\label{specialtype}
 Let $TM=\Horiz\oplus\Verti$ be a canonical $\g$-splitting. If $(M,g,\tau)$ is indecomposable, $\Verti\neq0$ and $\g$ acts trivially on $\Verti$, then the horizontal part $\tau^\Horiz$ vanishes.
\end{lem}

\begin{bem}
 Initially, the analogous statement to \cite[Lem~6.2]{CMS} would require that $(M,g,\tau)$ is \emph{$\g$-indecomposable} in the sense that there exists no nontrivial splitting of the tangent bundle $TM=T_1\oplus T_2$ into orthogonal, $\g$-invariant distributions such that $\tau$ splits as $\tau=\tau_1+\tau_2$ for some $\tau_i\in\Lambda^3T_i$. But in fact, all notions of $\g$-decomposability for intermediary subalgebras $\hol(\nabla^\tau)\subseteq\g\subseteq\stab(\tau)$ are equivalent. Clearly, $\g$-decomposability implies $\g'$-decomposablility for $\hol(\nabla^\tau)\subseteq\g'\subseteq\g\subseteq\stab(\tau)$, and by \cite[Lem.~3.2]{CMS}, $(M,g,\tau)$ is $\hol(\nabla^\tau)$-decomposable if and only if it is locally isometric to a Riemannian product $(M_1\times M_2,g_1+g_2,\tau_1+\tau_2)$. To close the circle, we need to observe that a Riemannian product is $\stab(\tau)$-decomposable. This is achieved by way of the following lemma.
\end{bem}

\begin{lem}
 Let $T=T_1\oplus T_2$ be a direct sum of finite-dimensional vector spaces, and let $\tau=\tau_1+\tau_2\in\Lambda^kT$ for $\tau_i\in\Lambda^kT_i$, where $k\geq3$. Then
 \[\stab_{\gl(T)}(\tau)=\stab_{\gl(T_1)}(\tau_1)\oplus\stab_{\gl(T_2)}(\tau_2).\]
\end{lem}
\begin{proof}
 Of course, if $A_i\in\stab_{\gl(T_i)}(\tau_i)$, then $(A_1\oplus A_2)_\ast\tau=(A_1)_\ast\tau_1+(A_2)_\ast\tau_2=0$. Conversely, let $A\in\gl(T)$ such that $A_\ast\tau=0$. Then we have
 \[A_\ast\tau_i\in T\otimes\Lambda^{k-1}T_i\subset\Lambda^kT,\qquad i=1,2,\]
 but also $A_\ast\tau_1+A_\ast\tau_2=0$. Since these two vector subspaces of $\Lambda^kT$ intersect trivially if $k\geq 3$, we conclude $A_\ast\tau_i=0$.
\end{proof}

\begin{bem}
To the best of our knowledge it is an open question whether there exists an indecomposible geometry with parallel skew torsion $(M,g,\tau)$ with a canonical $\g$-splitting for with $\Verti\neq0$ and $\tau^\Horiz\neq0$.
\end{bem}

It may seem as if one would lose information by taking $\g$ to be strictly larger than $\hol(\nabla^\tau)$ since the decomposition of the tangent space $\R^n$ may get coarser, and the vertical space smaller. However, there are some definite advantages. First of all, as already remarked in \cite{CMS}, it is in general not easy to determine the standard decomposition explicitly since $\hol(\nabla^\tau)$ is not always known. Second, as we will see later on, many situations require us to consider a local submersion based on an action of some fixed Lie algebra $\g$ which may be larger than $\hol(\nabla^\tau)$. Third, focusing on the action of $\stab(\tau)$ yields the most satisfactory analogue of the de Rham decomposition theorem for geometries with parallel skew torsion, in so far that it guarantees irreducibility of the factors of the base:

\begin{satz}
 Let $(M,g,\tau)$ be a geometry with parallel skew torsion, and $\pi: M\localmap N$ a canonical $\g$-submersion. Then the base is locally decomposable into geometries with parallel skew torsion $(N_\alpha,g_\alpha,\sigma_\alpha)$,
 \[(N,g_N,\sigma)\cong\prod_\alpha(N_\alpha,g_\alpha,\sigma_\alpha),\qquad\tau^{\Horiz}=\sum_\alpha\pi^\ast\sigma_\alpha,\]
 which are irreducible under their respective stabilizer algebra $\stab(\sigma_\alpha)$.
\end{satz}
\begin{proof}
Let $\h=\bigoplus_\alpha\h_\alpha$ be the horizontal part in the definition of the canonical $\g$-splitting. By \ref{tausplit} and \ref{base}, the base $(N,g_N,\sigma)$ of the corresponding local submersion $\pi$ is \emph{decomposable} and thus a product of geometries with parallel skew torsion $(N_\alpha,g_\alpha,\sigma_\alpha)$ by \cite[Lem.~3.2]{CMS}. Moreover, the 3-forms $\sigma_\alpha$ are related to $\tau$ by $\tau^{\Horiz_\alpha}=\pi^\ast\sigma_\alpha$.

It remains to show that $\stab(\sigma_\alpha)$ acts irreducibly on the tangent space $\h_\alpha$ of $N_\alpha$ for every $\alpha$. Consider the projection
\[\fs_\alpha:=\pr_{\so(\h_\alpha)}\g=\im(\g\to\so(\h_\alpha)).\]
Clearly, $\fs_\alpha$ acts irreducibly on $\h_\alpha$. Moreover, since $\g$ preserves each $\h_\alpha$, it has to stabilize each component $\tau^{\Horiz_\alpha}$ separately, so
\[\fs_\alpha\subseteq\stab_{\so(\h_\alpha)}(\tau^{\Horiz_\alpha})=\stab(\sigma_\alpha).\]
Hence $\stab(\sigma_\alpha)$ must also act irreducibly on $\h_\alpha$.
\end{proof}
Put differently, any geometry with parallel skew torsion is locally decomposable into \emph{stabilizer-irreducible} geometries with parallel skew torsion, \emph{possibly after passing to the base of a local submersion} with locally naturally reductive homogeneous fibers.

Note also that a stabilizer-irreducible geometry with parallel skew torsion needs not be de Rham irreducible if the torsion vanishes -- but this is simply the Riemannian case. In the following section we classify the cases where the torsion does not vanish, and in particular characterize the cases where $\hol(\nabla^\tau)\subsetneq\stab(\tau)$.

\section{Irreducible stabilizer actions}

In his PhD thesis, Cleyton gave a rough classification of the stabilizer-irreducible geometries with parallel, skew and nonvanishing torsion \cite[Thm.~I.3.14]{cleyton}. Later, Cleyton--Swann refined this classification and removed the assumption of skew-symmetric torsion, while specializing to irreducible holonomy action \cite[Thm~5.14]{CS}. Our aim is to show that essentially the same classification can be achieved by considering an arbitrary intermediate Lie algebra $\hol(\nabla^\tau)\subseteq\h\subseteq\stab(\tau)$ acting irreducibly on the tangent space. (Note that this intermediate algebra was denoted $\g$ in the previous section.)

More pressingly, there turns out to be a gap in both classifications, occurring in dimension $3$ (and contradicting Cleyton--Swann's conclusion that such a geometry is always Einstein!), plus another small gap in the symmetric setting where $g$ is flat. In order to fix the first gap, we need to revisit a technical lemma \cite[Lem.~5.13]{CS}. We give a correction to the statement in Lemma~\ref{curvg1} together with a conceptual, Lie-theoretic proof.

First, we need to recall a few definitions and facts. Let $V$ be a finite-dimensional Euclidean vector space. Consider the well-known $\O(V)$-invariant decomposition
\[\Sym^2\Lambda^2V=\Curv\oplus\Lambda^4V,\]
where $\Curv$ is the space of \emph{algebraic curvature tensors} on $V$. The projection to $\Lambda^4V$ is given by $\frac{1}{3}b$, where $b: \Sym^2\Lambda^2V\to\Lambda^4V$ is the \emph{Bianchi map}, which involves a cyclic sum in the first three arguments,
\[b(R)(X,Y,Z,W):=R(X,Y,Z,W)+R(Y,Z,X,W)+R(Z,X,Y,W).\]
For any subalgebra $\h\subseteq\so(V)$, we may view its space $\Sym\h$ of symmetric endomorphisms as a subspace of $\Sym^2\Lambda^2V$, and thus define
\[\Curv(\h):=\ker(b: \Sym\h\to\Lambda^4V)=\Curv\cap\Sym\h\]
as the space of \emph{algebraic curvature tensors with values in $\h$}. This representation of $\h$ has been instrumental in studying holonomy representations \cite{berger,CS,schwachh} as well as in the proof of Lemma~\ref{verthom}.

We remark that if we denote $\tau^2_{X,Y}Z:=\tau_Z\tau_XY$, then we have $(\tau_X)_\ast\tau=X\intprod b(\tau^2)$ and \eqref{curvdiff} may more elegantly be rewritten as
\begin{equation}
R^g=R^\tau+\tau^2+b(\tau^2).\label{curvdiff2}
\end{equation}

\subsection{The local classification}

First, we state the promised correction to \cite[Lem.~5.13 (i)]{CS}.

\begin{lem}
\label{curvg1}
 Let $V$ be a finite-dimensional Euclidean vector space underlying a faithful irreducible representation of a Lie algebra $\h\subseteq\so(V)$. If $(V\otimes\h)^\h\neq0$, then $\h$ is simple and $V\cong\h$. If $\h$ has rank at least $2$, then $\Curv(\h)\cong\R$, while for $\h=\so(3)$ we have $\Curv(\h)=\Curv\cong\Sym^2\R^3$.

 Moreover, if $\tau\in(V\otimes\h)^\h\subseteq(V\otimes\so(V))^\h$, then $\tau$ is a multiple of the canonical 3-form given by the Lie bracket of $\h$.
\end{lem}
\begin{proof}
 We write $\Hom_\h$ and $\End_\h$ for $\h$-equivariant homomorphisms (resp.~endomorphisms) of $\h$-modules. Since $V$ is irreducible, $\Hom_\h(V,\h)\cong(V\otimes\h)^\h\neq0$ means that there is a simple ideal $\h_0$ of $\h$ such that $\h_0\cong V$. Write $\h=\h_0\oplus\h_1$. Now $\h_0$ acts on $V$ by its adjoint representation, and the action of $\h_1$ commutes with it, that is $\h_1\subset\End_{\h_0}V$. However, since $\h_0$ preserves the inner product on $V$, it is a compact Lie algebra and thus its adjoint representation is of real type, i.e.~$\End_{\h_0}\h_0=\R\Id$. Since $\h_1\subset\End_{\h_0}V=\R\Id$ consists of skew-symmetric endomorphisms, it follows that $\h_1=0$. Hence $\h$ is simple.

 Since $\End_{\h}\h\cong(V\otimes\h)^\h\subseteq(V\otimes\so(V))^\h$ is one-dimensional, any element is a multiple of the canonical 3-form $\langle[\,\cdot\,,\cdot\,],\cdot\,\rangle$.

 If the rank of $\h$ is $1$, then $\h=\so(3)$. In this case, $\h\cong\Lambda^2\h$, and the condition $b(A)=0$ is vacuous since $\Lambda^4\R^3=0$. Thus $\Curv(\h)=\Curv\cong\Sym^2\so(3)$.

 Lastly, in the case where the rank of $\h$ is at least two, the statement $\Curv(\h)\cong\R$ was shown in \cite[p.~48]{schwachh}.
\end{proof}

\begin{bem}
 The first part of Lemma~\ref{curvg1}, namely the conclusion that $V\cong\h$, is reminiscent of the Skew-Torsion Holonomy Theorem by Olmos--Reggiani \cite{skewhol}. The substantial difference is that this theorem assumes that $\tau$ is a 3-form and that $\h\neq\so(V)$ and concludes that $\tau$ is $\h$-invariant, while Lemma~\ref{curvg1} assumes invariance and proves the total skew-symmetry of $\tau$.
\end{bem}

For convenience we restate \cite[Lem.~5.13 (ii)]{CS}. Due to \cite[Prop.~4.10]{CS}, the proof is really just a case-by-case check.

\begin{lem}
\label{curvg2}
 Let $V$ be a finite-dimensional Euclidean vector space underlying a faithful irreducible representation of a Lie algebra $\h\subseteq\so(V)$. If $(V\otimes\h^\perp)^\h\neq0$, where $\h^\perp$ is the orthogonal complement of $\h$ in $\so(V)$, and $\Curv(\h)\neq0$, then $(\h,V)$ is either $(\su(3),\C^3)$ or $(\g_2,\R^7)$. In both cases $\Curv(\h)$ is an irreducible representation not isomorphic to $\R$ or $V$. Moreover, $(V\otimes\h^\perp)^\h\subseteq(\Lambda^3V)^\h$.
\end{lem}

We are now ready to state and prove the correction and generalization of \cite[Thm.~5.14]{CS}. Note that as in \cite{CS}, the skew-symmetry of the torsion is a consequence instead of an assumption.

\begin{satz}
\label{csextended}
Let $(M,g)$ be a Riemannian manifold carrying a connection $\nabla$ with parallel nonzero torsion, and $\h$ any Lie algebra such that $\hol(\nabla)\subseteq\h\subseteq\stab(\tau)$ and $\h$ acts irreducibly on the tangent space of $M$. Then the torsion is skew-symmetric, i.e.~$\nabla=\nabla^\tau$ for some $\tau\in\Omega^3(M)$, and $(M,g,\tau)$ belongs to one of the following cases:
\begin{enumerate}[\upshape(a)]
\item $(M,g)$ is locally isometric to a non-symmetric isotropy irreducible homogeneous space $G/H$ with $\h=\mathrm{Lie}(H)$, and $\nabla$ is the canonical connection, cf.~\ref{natred}.
\label{casea}
\item $(M,g)$ is locally isometric to one of the irreducible symmetric spaces $(H\times H)/H$ or $H^\C/H$ with $\h=\mathrm{Lie}(H)$, or to the Euclidean vector space $\h$, and $\tau$ is any nonzero multiple of the canonical 3-form on $\h$.
\label{caseb}
\item $(M,g)$ is a Gray manifold or nearly parallel $G_2$, and $\nabla$ is the characteristic $\SU(3)$- or $\rmG_2$-connection, cf.~\ref{nearlykaehler} and \ref{nearlyg2}.
\label{casec}
\item $\dim M=3$ and $\tau$ is a multiple of $\vol_g$.
\label{cased}
\end{enumerate}
\end{satz}
\begin{proof}
 As the major part of the work has already been done in \cite{CS}, we give an account of the proof of \cite[Thm.~5.14]{CS} and point out the necessary modifications, corrections, and additions.

 \begin{enumerate}[(a)]
 \item
 First, assume that $\Curv(\h)=0$. Then also $\Curv(\hol(\nabla))=0$, and \cite[Lem~5.6]{CS} implies that $(M,g)$ is an Ambrose--Singer manifold. The regularity argument in the proof of \cite[Prop.~5.12]{CS} works analogously for $\h$ instead of $\hol(\nabla)$. Thus we conclude that $(M,g)$ is locally isometric to an isotropy irreducible homogeneous space $G/H$, where $H$ is a Lie group with Lie algebra $\h$, and that $\nabla$ is its canonical connection.

 The pair $(G,H)$ cannot be symmetric: indeed, suppose the Lie algebra of $\g$ splits as $\g=\h\oplus\m$ with $[\m,\m]\subseteq\h$, then the inclusion $\h\to\so(\m)$ and the projection $\Lambda^2\m\to\h$ are both given in terms of the Lie bracket, and the condition $b(A)=0$ for $A\in\Sym\h$ reduces to
 \[[A[X,Y],Z]+[A[Y,Z],X]+[A[Z,X],Y]=0,\qquad\forall X,Y,Z\in\m.\]
 This is however satisfied for $A=\Id_\h$ thanks to the Jacobi identity, which shows that $\Curv(\h)\neq0$, contradicting the assumption.
 \end{enumerate}

 Let us now assume that $\Curv(\h)\neq0$. Since $\tau\in(\R^n\otimes\so(n))^\h$, at least one of the spaces $(\R^n\otimes\h)^\h$ or $(\R^n\otimes\h^\perp)^\h$ must be nontrivial, where $\h^\perp$ is the orthogonal complement of $\h$ in $\so(n)$. It follows from Lemmas~\ref{curvg1} and~\ref{curvg2} that $\tau$ is a 3-form.

 \begin{enumerate}[resume*]
 \item
 If we assume that $b(\tau^2)=0$, then it follows from \eqref{curvdiff2} that $b(R^\tau)=0$. Moreover $\tau_X$ annihilates $\tau$ for any $X\in\R^n$, thus $\tau\in\R^n\otimes\stab(\tau)$. This means $\tau^2$ is an algebraic curvature tensor with values in $\stab(\tau)$, i.e.~$\tau^2\in\Curv(\stab(\tau))$, and in particular $\Curv(\stab(\tau))$ contains a trivial submodule. Lemma~\ref{curvg2} then implies that $(\R^n\otimes\stab(\tau)^\perp)^{\stab(\tau)}=0$, hence we must have $(\R^n\otimes\stab(\tau))^{\stab(\tau)}\neq0$, and by Lemma~\ref{curvg1} $\stab(\tau)$ is simple and $\R^n\cong\stab(\tau)$. Thus, under the action of $\h\subset\stab(\tau)$, $\h$ itself would form an invariant subspace of $\R^n\cong\stab(\tau)$. Since $\h$ is assumed to act irreducibly on $\R^n$, and $\h\neq0$ because $\Curv(\h)\neq0$, it follows that $\h=\stab(\tau)$.

 Assuming from now on $\rank\h\geq2$, Lemma~\ref{curvg1} guarantees that $\Curv(\h)\cong\R$. Since $b(R^\tau)=0$ and $R^\tau$ takes values in $\hol(\nabla^\tau)\subseteq\h$, we have $R^\tau\in\Curv(\h)$, and it follows that $R^\tau=\kappa\tau^2$ for some function $\kappa: M\to\R$. As in \cite{CS}, we may calculate using \eqref{curvdiff2} that
 \[\scal_g=2(1+\kappa)\|\tau\|^2.\]
 Since $\h$ acts irreducibly on $\R^n$, Schur's Lemma implies that $(\Sym^2_0\R^n)^\h=0$. Together with $\Curv(\h)\cong\R$, we may apply \cite[Thm.~5.3]{CS} to find that $(M,g)$ is Einstein. In particular $\scal_g$ is constant. Since $\tau$ is parallel with respect to the metric connection $\nabla^\tau$, it has constant norm. Thus also $\kappa$ is constant, and we have
 \[\nabla^g_XR^g=(1+\kappa)(\nabla^\tau_X\tau^2-(\tau_X)_\ast\tau^2)=0.\]
 Thus $(M,g)$ is locally symmetric with tangent space $\h$. Let $G/K$ denote the symmetric model, where $\k=\im R^g\subseteq\Lambda^2\R^n$. Since $\nabla^\tau\tau=0$ and $b(\tau^2)=0$, we also have $\nabla^g\tau=0$. That is, $\tau$ lifts to an invariant 3-form on $G/K$. Hence $K$ stabilizes $\tau$, i.e.~$\k\subseteq\h=\stab(\tau)$. By $R^\tau=\kappa\tau^2$ and \eqref{curvdiff2}, $R^g$ is also invariant under $\h$. Thus $\k=\im R^g$ is an $\h$-invariant subspace of $\h$, i.e.~an ideal. Since $\h$ is simple, this means either $\k=0$ or $\k=\h$. In the first case, $g$ is flat, and in the second case $\g=\h\oplus\h$ (as a representation of $\h$). It is well known that the only two symmetric spaces with this property are $(H\times H)/H$ and $H^\C/H$.

 \item
 Next, if $\Curv(\h)\neq0$ and $b(\tau^2)\neq0$, then the projection of $\tau$ to $\R^n\otimes\stab(\tau)^\perp$, and in particular to $\R^n\otimes\h^\perp$, is nontrivial. Then Lemma~\ref{curvg2} implies that $\h=\stab(\tau)$ is one of $\su(3)$ or $\g_2$, and the representation on the tangent space is $\C^3$ or $\R^7$, respectively. Thus $(M,g,\tau)$ is a geometry with parallel skew torsion whose stabilizer is $\su(3)\subset\so(6)$ or $\g_2\subset\so(7)$. By Lemma~\ref{concludeNK6}, resp.~Lemma~\ref{concludeNG2}, $(M,g)$ is a Gray manifold or a nearly parallel $\rmG_2$-manifold.

 \item
 Lastly, if $\rank\h=1$, then $\dim M=3$. Again $\tau$ has constant norm since $\nabla^\tau\tau=0$, so it must be a constant multiple of the volume form $\vol_g$ (and in particular, $M$ is orientable).\qedhere
 \end{enumerate}
\end{proof}

\subsection{Stabilizer versus holonomy}

The above generalization of \cite[Thm.~5.14]{CS} raises the question how many more cases it actually catches. In other words,
\begin{itemize}
 \item which of the above situations allow for $\hol(\nabla^\tau)$ to be properly contained in $\stab(\tau)$, or even for the existence of a proper intermediate subalgebra $\h$?
 \item can we characterize the stabilizer-irreducible geometries with parallel skew torsion for which the holonomy representation is reducible?
\end{itemize}
These questions drive the remainder of the article.

\begin{prop}
 In cases \ref{casea} and \ref{caseb} above we actually have $\hol(\nabla^\tau)=\stab(\tau)$, except for the Berger space $\SO(5)/\SO(3)_{\mathrm{irr}}$, and the flat $(\pm)$-connections on $(H\times H)/H$.
\end{prop}
\begin{proof}
 \begin{enumerate}[(a)]
  \item Let $G/H$ be a non-symmetric isotropy irreducible homogeneous space, $\nabla$ its canonical connection, and $\g=\h\oplus\m$ its reductive decomposition. Note that $\h\subseteq\so(\m)$, so it follows from \eqref{holhom} that
  \[\hol(\nabla)\oplus\m=\m+[\m,\m],\]
  and this is in fact an ideal in $\g$. But by \cite[Thm.~1.1]{wolf}, $G$ is simple, so this ideal must equal $\g$, and this only possible if $\hol(\nabla)=\h$.

  Suppose that $\h\subsetneq\stab(\tau)$. If $\Curv(\stab(\tau))=0$, we may apply Theorem~\ref{csextended} and repeat the above argument to conclude that $\stab(\tau)=\hol(\nabla)$. On the other hand, if $\Curv(\stab(\tau))\neq0$, then we fall in case \ref{casec} of Theorem~\ref{csextended} for the stabilizer action; that is, the isotropy representation is the restriction of $\C^3$ or $\R^7$ to $\h\subset\su(3)$ or to $\h\subset\g_2$, respectively. Since no proper subalgebra of $\su(3)$ acts irreducibly on $\C^3$, this case is ruled out; and the only proper subalgebra of $\g_2$ that acts irreducibly on $\R^7$ is $\h=\so(3)_{\mathrm{irr}}$, as defined in Example~\ref{nearlyg2}. By the classification of Wolf \cite{wolf}, $G/H$ must be the Berger space $\SO(5)/\SO(3)_{\mathrm{irr}}$.
  \item Assume now that $b(\tau^2)=0$. Then $M$ is either $(H\times H)/H$, $\h$, or $H^\C/H$, and
  \[\tau_XY=-t[X,Y],\qquad t\neq0,\]
  for a suitable identification $\h\cong\m$ (cf.~Example~\ref{natred}), while the curvature is given by
  \[R^\tau(X,Y)=-s\ad([X,Y]),\qquad s=\begin{cases}
                                       1-t^2,&M=(H\times H)/H,\\
                                       -t^2&M=\h,\\
                                       -1-t^2,&M=H^\C/H.
                                      \end{cases}\]
  Thus $\hol(\nabla^\tau)=\im(R^\tau)=\h$, except when $s=0$, which happens only on the space $(H\times H)/H$ for $t=\pm1$. Moreover, $\stab(\tau)$ is the Lie algebra of derivations of $\h$, and since $\h$ is simple, this coincides with $\h$.
 \end{enumerate}
\end{proof}

Before we carry on with the cases \ref{casec} and \ref{cased} of Theorem~\ref{csextended}, we need to state a few preparatory results.

\begin{lem}
\label{nk6parallel}
 Let $(M^6,g,J)$ be a Gray manifold and $\nabla^\tau$ its canonical Hermitian connection. Then there do not exist any nontrivial $\nabla^\tau$-parallel vector fields on $M$.
\end{lem}
\begin{proof}
 Let $\xi\in\X(M)$ be a $\nabla^\tau$-parallel vector field, that is
 \[\nabla^g_X\xi=-\tau_X\xi,\qquad X\in TM.\]
 In particular $\xi$ is Killing and $d\xi=2\nabla^g\xi=2\tau_\xi$. Moreover, $\tau_\xi$ is $\nabla^\tau$-parallel since $\tau$ and $\xi$ are. Using \eqref{extaction} and \cite[(8)]{MS}, one may calculate that
 \begin{align*}
  0&=\frac{1}{2}d^2\xi=\sum_ie_i\wedge\nabla^g_{e_i}\tau_\xi=-\sum_ie_i\wedge(\tau_{e_i})_\ast\tau_\xi=-\xi\intprod\sum_i\tau_{e_i}\wedge\tau_{e_i}=-2\xi\intprod(\omega\wedge\omega),
 \end{align*}
 where $\omega$ is the Kähler form (identified with $J$ using the metric). Since $\omega\wedge\omega$ is nondegenerate, it follows that $\xi=0$.
\end{proof}

\begin{lem}[\cite{g2paralleltorsion}, Prop.~3.1]
\label{ng2parallel}
 Let $(M^7,g,\varphi)$ be a nearly parallel $\rmG_2$-manifold and $\nabla^\tau$ its canonical $\rmG_2$-connection. Then there do not exist any nontrivial $\nabla^\tau$-parallel vector fields on $M$.
\end{lem}
\begin{proof}
 We prove this by contradiction. Let $X\neq0$ be a $\hol(\nabla^\tau)$-invariant element of $\R^7$ corresponding to a $\nabla^\tau$-parallel vector field. Let $\h:=\stab_{\g_2}(X)$ be the subalgebra of $\g_2$ that annihilates $X$. Then $\hol(\nabla^\tau)\subseteq\h$. It is well-known that the stabilizer algebra in $\g_2$ of any nonzero vector in $\R^7$ is conjugate to $\su(3)\subset\g_2$, so the holonomy representation splits into $\R^7=\R\oplus\C^3$. This is precisely the canonical $\h$-splitting, where $\Verti=\R$ and $\Horiz=\C^3$. From Lemma~\ref{specialtype} it follows now that $\tau^\Horiz=0$, and thus $\tau\in\Lambda^2\Horiz\otimes\Verti$. But this implies
 \[\tau_X\wedge\tau_X\wedge\tau=0\]
 for any $X\in\Horiz$, which is not possible for a $3$-form stabilized by $\rmG_2$.
\end{proof}

\begin{lem}
\label{ascompact}
 Let $(M,g,\tau)$ be a naturally reductive Ambrose--Singer manifold, and let $\g=\h\oplus\m$ be its transvection algebra. If $\tau\in\m\otimes\h^\perp\subseteq\m\otimes\so(\m)$ and
 \[\Ric^g(X,X)+|\tau_X|^2>0\qquad\forall X\in\m,\ X\neq0,\]
 then $\g$ is compact and semisimple.\footnote{Here we use the endomorphism norm, which is given by $|\alpha|^2=\sum_{i}|\alpha(e_i)|^2$ for $\alpha\in\End\m$.}
\end{lem}
\begin{proof}
 It suffices to show that the Killing form $B$ of $\g$ is negative definite. First, we show that $B(\m,\h)=0$. Extend $g$ to the natural inner product on $\g\subseteq\m\oplus\Lambda^2\m$, and let $(e_i)$ and $(f_j)$ be orthonormal bases of $\m$ and $\h$, respectively. Then for any $A\in\h$ and $X\in\m$, we have
 \[B(X,A)=\tr_\g(\ad(X)\ad(A))=\sum_i\langle[X,[A,e_i]],e_i\rangle+\sum_j\langle[X,[A,f_j]],f_j\rangle.\]
 The second term vanishes because $[\h,\h]\subseteq\h$ and $[\h,\m]\subseteq\m\perp\h$. For the first term, we use natural reductivity and the definition \eqref{transvec} of the bracket on $\g$ to obtain
 \[\sum_i\langle[X,[A,e_i]],e_i\rangle=\sum_i\langle Ae_i,2\tau_Xe_i\rangle=4\langle A,\tau_X\rangle_{\Lambda^2\m},\]
 which vanishes since by assumption, $\tau_X\in\h^\perp$.

 It remains to show that the restrictions of $B$ to $\h$ and $\m$ are negative definite. For $A\in\h\subseteq\so(\m)$ we have
 \begin{align*}
  B(A,A)&=\sum_i\langle[A,[A,e_i]],e_i\rangle+\sum_j\langle[A,[A,f_j]],f_j\rangle\\
  &=-\sum_i|Ae_i|^2-\sum_j|[A,f_j]|^2=-|A|^2-\sum_j|[A,f_j]|^2,
 \end{align*}
 which is clearly negative definite. Finally, for $X\in\m$, we have by \cite[Cor.~7.38]{besse} and \eqref{transvec}
 \begin{align*}
  B(X,X)&=-2\Ric^g(X,X)-\sum_i|[X,e_i]_\m|^2+\frac12\sum_{i,j}\langle[e_i,e_j],X\rangle^2\\
  &=-2\Ric^g(X,X)-4|\tau_X|^2+2|\tau_X|^2
 \end{align*}
 and this is negative if and only if $\Ric^g(X,X)+|\tau_X|^2>0$.
\end{proof}

The following result was already proven by Nagy under the additional assumption of completeness \cite[Cor.~3.1]{nagy}.

\begin{prop}
\label{nk6hol}
 Let $(M^6,g,\tau)$ be a geometry with parallel skew torsion such that $\hol(\nabla^\tau)\subsetneq\stab(\tau)=\su(3)$. Then either $(M,g,\tau)$ is locally isomorphic to the homogeneous Gray manifold $S^3\times S^3$, or the holonomy representation is reducible as a complex representation.
\end{prop}
\begin{proof}
 If $\hol(\nabla^\tau)\subsetneq\su(3)$, then $\hol(\nabla^\tau)$ is contained in a maximal proper subalgebra of $\su(3)$. Up to conjugation, these are $\fs(\u(2)\oplus\u(1))$ and $\so(3)$. In the first case, the holonomy representation splits into $\C^3=\C^2\oplus\C$, which are preserved by the almost complex structure $J$.

 In the second case, the holonomy representation splits as $\C^3=\R^3\oplus\R^3$ according to the real structure on $\C^3$ that $\so(3)\subset\su(3)$ preserves -- that is, the two summands are interchanged by $J$. Since the two summands are equivalent as representations of $\so(3)$, they are both vertical with respect to the canonical $\so(3)$-splitting. Lemma~\ref{verthom} then implies that $(M,g,\tau)$ is an Ambrose--Singer manifold.

 We must have $\hol(\nabla^\tau)=\so(3)$. Indeed, if $\hol(\nabla^\tau)\subsetneq\so(3)$, then the holonomy representation $\C^3=\R^3\oplus\R^3$ has to have at least two trivial summands and thus $\nabla^\tau$-parallel vector fields, which is impossible by Lemma~\ref{nk6parallel}.

 Let $\h:=\so(3)$, $\m:=\C^3$ the holonomy representation, and let $\g:=\h\oplus\m$ be the transvection algebra of the Ambrose--Singer manifold $(M,g,\tau)$. As $(\C^3\otimes\su(3))^{\su(3)}=0$, it follows that $\tau\in(\m\otimes\su(3)^\perp)^{\su(3)}\subset(\m\otimes\h^\perp)^{\h}$. Since $(M,g,\tau)$ is a Gray manifold, it is Einstein with positive scalar curvature, so in particular the assumptions of Lemma~\ref{ascompact} are satisfied, and we conclude that $\g$ is compact and semisimple. For dimensional reasons, the Lie algebra $\g$ must then be isomorphic to $\so(3)\oplus\so(3)\oplus\so(3)$.

 The only possible embeddings $\h=\so(3)\hookrightarrow\g$ as a subalgebra are given by
 \[X\mapsto(aX,bX,cX),\qquad X\in\so(3),\]
 where $a,b,c\in\{0,1\}$. If one of them was $0$, then the isotropy representation $\m\cong\g/\h$ would have a trivial summand; and since this is not the case, $\h$ must be the diagonal subalgebra. In particular, the associated subgroup $\diag(S^3)\subset S^3\times S^3\times S^3$ is closed, and thus $(M,g,\tau)$ is locally isometric to the Gray manifold $S^3\times S^3=\frac{S^3\times S^3\times S^3}{\diag(S^3)}$.
\end{proof}

\begin{prop}
\label{ng2hol}
 Let $(M^7,g,\tau)$ be a geometry with parallel skew torsion such that $\hol(\nabla^\tau)\subsetneq\stab(\tau)=\g_2$. Then it is either locally isomorphic to the Berger space, or $\hol(\nabla^\tau)\subseteq\so(4)\subset\g_2$ and the holonomy representation is reducible.
\end{prop}
\begin{proof}
 Again, $\hol(\nabla^\tau)$ must be contained in a maximal proper subalgebra of $\g_2$, and up to conjugation, these are $\so(3)_{\mathrm{irr}}$, $\so(4)$, and $\su(3)$.

 For the first case, we note that
 \[\so(7)\cong\g_2\oplus\R^7\cong\so(3)_{\mathrm{irr}}\oplus\R^7\oplus\R^{11}\]
 under $\so(3)_{\mathrm{irr}}\subset\g_2$. Thus $(\R^7\otimes\so(3)_{\mathrm{irr}}^\perp)^{\so(3)_{\mathrm{irr}}}\neq0$, and since $\so(3)_{\mathrm{irr}}$ acts irreducibly on $\R^7$, Lemma~\ref{curvg2} yields that $\Curv(\so(3)_{\mathrm{irr}})=0$. By the proof of Theorem~\ref{csextended}, $(M,g,\tau)$ must locally be an isotropy irreducible space, and we may again invoke the classification of Wolf \cite{wolf} to see that it is the Berger space.

 In the second case where $\hol(\nabla^\tau)\subseteq\so(4)$, the holonomy representation splits as $\R^7=\R^3\oplus\R^4$. The action of $\so(4)\cong\so(3)\oplus\so(3)$ is by the adjoint representation of one of the $\so(3)$-factors on $\R^3$, and by the standard representation on $\R^4$.

 Finally, assume that $\hol(\nabla^\tau)\subseteq\su(3)$. Then the holonomy representation splits as $\R^7=\R\oplus\C^3$. In particular, there exists a nontrivial $\nabla^\tau$-parallel vector field. Hence this case is ruled out by Lemma~\ref{ng2parallel}.
\end{proof}

\begin{bem}
 In \cite[Thm.~8.1]{g2paralleltorsion}, the Berger space had already been characterized as the unique complete, simply connected and \emph{cocalibrated} $\rmG_2$-manifold with parallel skew torsion such that $\hol(\nabla^\tau)=\so(3)_{\mathrm{irr}}$. This is slightly more general than Prop.~\ref{ng2hol}, since it goes beyond the \emph{nearly parallel $\rmG_2$} condition.
\end{bem}

\begin{prop}
\label{3dim1}
 If $(M^3,g,\tau)$ is a geometry with parallel skew torsion such that $\hol(\nabla^\tau)\subsetneq\so(3)$, then $(M,g)$ is Sasakian up to rescaling and locally fibers over a Riemann surface.
\end{prop}
\begin{proof}
 By Theorem~\ref{csextended} \ref{cased}, we have $\tau=t\,\vol_g$ for some $t\in\R$, and we may use the Hodge star to write $\tau_X=t\star\!X$ and $\tau_XY=t\star\!(X\wedge Y)$ for vector fields $X,Y$. Thus
 \[\tau_X\tau_YZ=t^2\star\!(X\wedge\star(Y\wedge Z))=-t^2(Y\wedge Z)X.\]
 In dimension $3$ the Bianchi map is identically zero, and so \eqref{curvdiff2} reduces to
 \begin{equation}
 R^g(X,Y)=R^\tau(X,Y)-t^2X\wedge Y.\label{curvdiffdim3}
 \end{equation}
 The only maximal proper subalgebra of $\so(3)$, up to conjugacy, is $\so(2)$. Suppose that $\hol(\nabla^\tau)\subseteq\so(2)$. Then the holonomy representation splits as $\R^3=\R^2\oplus\R$. In particular, there exists a $\hol(\nabla^\tau)$-invariant vector, and thus a $\nabla^\tau$-parallel vector field $\xi$, assumed to be of unit length. As in the proof of Lemma~\ref{nk6parallel}, $\xi$ is Killing and satisfies $d\xi=2\tau_\xi=2t\star\xi$. Since $\nabla^\tau d\xi=0$, we have
 \begin{align*}
  \nabla^g_Xd\xi&=-(\tau_X)_\ast d\xi=-2(\tau_X)_\ast\tau_\xi=-2X\intprod b(\tau^2)+2\tau_{\tau_X\xi}=2t^2X\wedge\xi.
 \end{align*}
 Thus, up to a rescaling of the metric $g$ by the factor $t^2$, the Sasaki condition \eqref{sasaki} is satisfied. The canonical $\so(2)$-splitting is given by $\Horiz=\R^2$ and $\Verti=\R$, so $(M,g)$ locally fibers over a surface $(N,g_N)$. We readily compute
 \[(\star\xi)^2X=\star(\xi\wedge\star(\xi\wedge X))=-(\xi\wedge X)\xi=-X,\qquad X\in\Horiz,\]
 thus $\star\xi$ defines an almost complex structure on the distribution $\Horiz$. Together with the above, Cartan's formula immediately implies that $\Lie_\xi(\star\xi)=0$, so $\star\xi$ projects to a complex structure $J$ on $N$. Hence $(N,g_N,J)$ is a Riemann surface.
\end{proof}

\begin{bem}
\label{3dim2}
 Finally, if $(M^3,g,\tau)$ is a geometry with parallel skew nonzero torsion such that $\hol(\nabla^\tau)=0$, i.e. $\nabla^\tau$ is flat, then \eqref{curvdiffdim3} implies that $R^g=-t^2\Id_{\Lambda^2}$. Hence $g$ has constant sectional curvature $t^2>0$, that is, $(M,g)$ is locally isometric to a round 3-sphere.
\end{bem}

\section{Gray manifolds with complex reducible holonomy}
\label{sec:nk6}

Let $(M^6,g,J)$ be a Gray manifold whose canonical connection $\nabla^\tau$ has complex reducible holonomy representation, that is, $\hol(\nabla^\tau)\subseteq\fs(\u(1)\oplus\u(2))$ (see Proposition~\ref{nk6hol}). Then the canonical $\fs(\u(1)\oplus\u(2))$-splitting is given by $TM=\Horiz\oplus\Verti$, where pointwise $\Horiz\cong\C^2$ and $\Verti\cong\C$. In particular it is $J$-invariant. The $3$-form $\tau=-\frac12J\circ\nabla^gJ$ is of type $(3,0)+(0,3)$, and since $\dim_\C\Horiz=2$ and $\dim_\C\Verti=1$, we necessarily have $\tau\in\Lambda^2\Horiz\otimes\Verti$.

The splitting $TM=\Horiz\oplus\Verti$ locally defines a Riemannian submersion with totally geodesic fibers over some manifold $(N^4,g_N)$. Our goal in this section is to prove that this local submersion is equivalent to the twistor fibration.

\begin{satz}
\label{nk6reducible}
If the canonical connection of a Gray manifold $(M,g,J)$ has complex reducible holonomy representation, i.e. $\hol(\nabla^\tau)\subseteq\fs(\u(1)\oplus\u(2))\subset\su(3)$, then it is locally isomorphic to the twistor space over an anti-self-dual Einstein $4$-manifold $(N,g_N)$.
\end{satz}

\begin{bem}
 It was first shown by Reyes-Carrión \cite{reyes} that the canonical connection of a nearly Kähler twistor space over an anti-self-dual Einstein $4$-manifold has holonomy contained in $\mathrm{S}(\U(1)\times\U(2))$. Theorem~\ref{nk6reducible} may be seen as a converse to that.

 Under the additional assumption that $(M,g)$ is complete, Belgun--Moroianu \cite{BM} have already shown that if $(M,g,\tau)$ has complex reducible holonomy, it is isometric to one of the homogeneous Gray manifolds $\CP^3$ or $F_{1,2}$ (see Remark~\ref{homNK6}), which are the twistor spaces over $S^4$ or $\CP^2$ with their respective standard metric. However, there are many more anti-self-dual Einstein $4$-manifolds which are \emph{not} complete, and their twistor spaces are captured by Theorem~\ref{nk6reducible}.

 A higher-dimensional version of Theorem~\ref{nk6reducible} is given by Stecker \cite[Thm.~4.5]{steckerNK}, stating that under some technical assumptions, any nearly Kähler manifold $M^{n+2}$ with a suitable $\nabla^\tau$-parallel splitting of the tangent bundle fibers locally over a quaternion-Kähler manifold $N^n$. Our formulation in dimension $n+2=6$ removes these assumptions.

 A version for Hermitian manifolds with parallel skew torsion and holonomy contained in $\U(m)\times\U(1)$ was proved by Alexandrov \cite[Thm.~7.1]{alexandrovhermitian}. Again, under suitable assumptions on the torsion, the space is locally isomorphic to a twistor space over a positive quaternion-Kähler manifold.
\end{bem}

\subsection{The twistor space over a 4-manifold}

First, let us give an account of the twistor construction. Let $(N^4,g_N)$ be an oriented Riemannian 4-manifold. Its \emph{twistor space} $Z$ is defined as the bundle of compatible almost complex structures:
\[Z:=\{(p,j)\,|\,p\in N,\ j\in\SO(T_pN),\ j^2=-\Id\}.\]
The corresponding bundle map
\[\pi_Z:\quad Z\longrightarrow N:\quad (p,j)\longmapsto p,\]
is called the \emph{twistor fibration}. One may also view the fiber bundle $Z\cong P\times_{\SO(4)}S^2$ as associated to the principal bundle $P\to N$ of oriented orthonormal frames, with fiber $S^2=\SO(4)/\U(2)$. Since almost complex structures are in particular skew-symmetric, $Z\subset\so(TN)=P\times_{\SO(4)}\so(4)$.

Let $\Verti_Z:=\ker d\pi_Z$ be the vertical distribution of the fibration $\pi_Z$. Its fiber over a point $(p,j)\in Z$ is identified with the tangent space $T_j\pi_Z^{-1}(p)$ of the fiber, that is
\[\Verti_{Z,(p,j)}=\{A\in\so(T_pN)\,|\,A\circ j+j\circ A=0\}.\]
The Levi-Civita connection $\nabla^{g_N}$ induces a connection on the bundle of skew-symmetric endomorphisms $\so(TN)$, which restricts to an Ehresmann connection on $Z$. Take $\Horiz_Z$ to be the horizontal distribution of this Ehresmann connection. That is, a curve $\gamma$ in $Z$ with $\gamma(0)=j$ is horizontal if for any $t$, $\gamma(t)$ is the $\nabla^{g_N}$-parallel transport of $j$ in $\so(TN)$ along the curve $\pi_Z\circ\gamma\big|_{[0,t]}$. We then have a splitting $TZ=\Verti_Z\oplus\Horiz_Z$, and the bundle map $d\pi_Z\big|_{\Horiz_Z}: \Horiz_Z\to TN$ is invertible.

Denote with $g_{S^2}$ the metric on the twistor fibers given by the restriction of the usual inner product on $\End TN$, i.e.
\[g_{S^2}(X,Y)=-\tr(XY),\qquad X,Y\in\Verti_Z\subset\so(TN).\]
It can be shown that on each fiber, $g_{S^2}$ is the round metric with Gaussian curvature $1$. Now, one may define a family of Riemannian metrics $(g_\lambda)_{\lambda>0}$ on $Z$ by setting
\[g_\lambda\big|_{\Horiz_Z\times\Horiz_Z}:=\pi_Z^\ast g,\qquad g_\lambda\big|_{\Verti_Z\times\Verti_Z}:=\lambda^{-1}g_{S^2},\qquad g_\lambda\big|_{\Horiz_Z\times\Verti_Z}:=0,\]
each making $\pi_Z$ into a Riemannian submersion.

Finally, let $J^\pm$ be the almost complex structures defined on $Z$ by
\begin{align*}
J^\pm_{(p,j)}(X)&=d\pi_Z\big|_{\Horiz_Z}^{-1}\circ j\circ d\pi_Z(X),&X&\in\Horiz_{Z,(p,j)},\\
J^\pm_{(p,j)}(j')&=\pm j\circ j',&j'&\in\Verti_{Z,(p,j)}.
\end{align*}
These are compatible with the metrics $g_\lambda$. It is well-known that if $(N,g_N)$ is Einstein with anti-self-dual Weyl curvature and positive scalar curvature, then $(Z,g_\lambda,J^+)$ is Kähler for $\lambda=\scal_{g_N}/3$, and $(Z,g_\lambda,J^-)$ is strictly nearly Kähler for $\lambda=2\scal_{g_N}/3$. Conversely, if $(Z,g_\lambda,J^+)$ is Kähler or $(Z,g_\lambda,J^-)$ is nearly Kähler, then $(N,g_N)$ is anti-self-dual Einstein with positive scalar curvature, and $\lambda=\scal_{g_N}/3$ or $\lambda=2\scal_{g_N}/3$, respectively \cite{FK,muskarov}.

\subsection{The isometry}

Let now $(M,g,J)$ be as in the beginning of~\ref{sec:nk6}, and $\pi: M\to N$ the canonical submersion. In order to prove Theorem~\ref{nk6reducible}, it suffices to give an isometry between $(M,g,J)$ and $(Z,g_\lambda,J^-)$ with $\lambda=2\scal_{g_N}/3$ that intertwines the almost complex structures. We define the smooth map
\begin{equation}
 F:\ (M,g,J)\to (Z,g_\lambda,J^-):\qquad p\longmapsto(\pi(p),d\pi\circ J_p^\Horiz\circ d\pi\big|_\Horiz^{-1}),\label{nkmap}
\end{equation}
where $J^\Horiz\in\End\Horiz$ is the horizontal part of $J\in\End TM$ (recall that $J$ preserves the distributions $\Horiz$ and $\Verti$). For the purpose of this section, we use the conventional normalization $\scal_g=30$. Then we show that for the choice $\lambda=16$, the map $F$ is a holomorphic isometry.

\begin{lem}
\label{nkisom1}
 The map $F: M\to Z$ defined in \eqref{nkmap} preserves the horizontal and vertical distributions, that is, $dF(\Horiz)=\Horiz_Z$ and $dF(\Verti)=\Verti_Z$.
\end{lem}
\begin{proof}
 First, let $V\in\Verti$. Then by \ref{vertdist},
 \[0=d\pi(V)=d(\pi_Z\circ F)(V),\]
 and thus $dF(V)\in\Verti_Z$.

 Now let $\gamma: [0,1]\to M$ be any horizontal curve, i.e.~$\dot\gamma(t)\in\Horiz$ for all $t$. Then $dF(\dot\gamma(0))=j'(0)$, where
 \[j(t)=F(\gamma(t))=(\pi(\gamma(t)),d\pi\circ J_{\gamma(t)}^\Horiz\circ d\pi\big|_\Horiz^{-1}).\]
 We intend to show that $j: [0,1]\to Z$ is parallel along $\pi\circ\gamma=\pi_Z\circ j$, and thus a horizontal curve. For any vector field $X\in\X(N)$, let $\tilde X$ denote its horizontal lift to $M$, i.e.
 \[\tilde X_p=d\pi_p\big|_\Horiz^{-1}X_{\pi(p)},\qquad p\in M.\]
 Suppose that $X$ is a $\nabla^{g_N}$-parallel vector field along $\pi\circ\gamma$. Since $\tau^\Horiz=0$, \ref{base} states that
 \[\nabla^\tau_{\dot\gamma}\tilde X=\widetilde{\nabla^{g_N}_{(\pi\circ\gamma)'}X}=0.\]
 Because $J$ and the distributions $\Verti,\Horiz$ are $\nabla^\tau$-parallel, so is the horizontal part $J^\Horiz$. Hence
 \[0=\nabla^\tau_{\dot\gamma}(J^\Horiz\tilde X)=\widetilde{\nabla^{g_N}_{(\pi\circ\gamma)'}(jX)}.\]
 Since $X$ was assumed to be parallel, it follows that $j$ is parallel along $\pi\circ\gamma$, hence $dF(\dot\gamma(0))=j'(0)\in\Horiz_Z$.
\end{proof}

Since both $\pi$ and $\pi_Z$ are Riemannian submersions and $\pi=\pi_Z\circ F$, we obtain:

\begin{kor}
\label{nkisom2}
 For any $X\in\Horiz$, $|X|^2=|dF(X)|^2$.
\end{kor}

\begin{lem}
\label{nkisom3}
 For any $V\in\Verti$, we have
 \[dF(V)=d\pi\circ 4J\tau_V\circ d\pi\big|_{\Horiz}^{-1}\]
 and $|V|^2=|dF(V)|^2$.
\end{lem}
\begin{proof}
 Let $V\in\X(M)$ be a vertical vector fields, and let $\Phi$ denote the flow of $V$. Since $\Verti$ is an integrable distribution by \ref{vertdist}, the flow of $V$ preserves the fibers of $\pi$, that is
 \[\pi\circ\Phi_t=\pi\qquad\forall t.\]
 Hence we may calculate at every $p\in M$ that
 \begin{align*}
 dF(V_p)&=\frac{d}{dt}\big|_{t=0}d\pi_{\Phi_t(p)}\circ J^\Horiz_{\Phi_t(p)}\circ d\pi_{\Phi_t(p)}\big|_\Horiz^{-1}\\
 &=\frac{d}{dt}\big|_{t=0}d\pi_{p}\circ(d\Phi_{-t})_{\Phi_t(p)}\circ J^\Horiz_{\Phi_t(p)}\circ(d\Phi_t)_{p}\circ d\pi_{p}\big|_\Horiz^{-1}\\
 &=d\pi_{p}\circ(\Lie_VJ^\Horiz)^\Horiz_p\circ d\pi_{p}\big|_\Horiz^{-1}.
 \end{align*}
 That is, $dF(V)=d\pi\circ(\Lie_VJ^\Horiz)^\Horiz\circ d\pi\big|_\Horiz^{-1}\in\so(TN)$. Further, using that$J$, $\Horiz$ and $\Verti$ are $\nabla^\tau$-parallel, that $\tau\in\Lambda^2\Horiz\otimes\Verti$, and that $\tau$ is a $(3,0)+(0,3)$-form, we find
 \begin{align*}
 (\Lie_VJ^\Horiz)^\Horiz(X)&=[V,JX]^\Horiz-J[V,X]^\Horiz\\
 &=(\nabla^g_V(JX))^\Horiz-(\nabla^g_{JX}V)^\Horiz-J(\nabla^g_VX)^\Horiz+J(\nabla^g_XV)^\Horiz\\
 &=(\nabla^\tau_VJ^\Horiz)(X)-(\nabla^\tau_{JX}V)^\Horiz+J\nabla^\tau_XV-2(\tau_V(JX))^\Horiz+2J(\tau_VX)^\Horiz\\
 &=4J\tau_VX.
 \end{align*}
 for any horizontal vector field $X\in\X(M)$. By \cite[Thm.~5.2 (i)]{gray}, we have
 \[|\tau_XY|^2=\frac{1}{4}|X|^2|Y|^2\qquad\text{if }X\perp Y,JY\]
 (where we note that our normalization corresponds to $\alpha=1$). Since $\pi$ is a Riemannian submersion, we conclude
 \[|dF(V)X|^2=|4J\tau_V\tilde X|^2=4|V|^2|X|^2\]
 for any $X\in TN$ with horizontal lift $\tilde X\in\Horiz$. Taking the trace,
 \[g_{S^2}(dF(V),dF(V))=-\tr(dF(V)^2)=16|V|^2.\]
 It follows that for $\lambda=16$, we have $|dF(V)|^2=|V|^2$ with respect to the metric $g_\lambda$.
\end{proof}

Together with Lemma~\ref{nkisom1} and Corollary~\ref{nkisom3}, this shows that $F: M\to Z$ is an isometry. It remains to show that it intertwines $J$ and $J^-$.

\begin{lem}
\label{nkisom4}
 $dF\circ J=J^-\circ dF$.
\end{lem}
\begin{proof}
Let $p\in M$. Since $F$ is an isometry, $dF_p$ is invertible. By definition,
\begin{align*}
J^\pm_{F(p)}&=(d\pi_Z)_{F(p)}\big|_{\Horiz_Z}^{-1}\circ d\pi_{p}\circ J_p^\Horiz\circ d\pi_{p}\big|_\Horiz^{-1}\circ(d\pi_Z)_{F(p)}\\
&=dF_{p}\circ J^\Horiz_p\circ (dF_p)^{-1}
\end{align*}
since $\pi_Z\circ F=\pi$. Thus $dF(JX)=J^\pm dF(X)$ for any $X\in\Horiz$.

Let now $V\in\Verti$. By Lemma~\ref{nkisom3}, and since $\tau$ is of type $(3,0)+(0,3)$,
\begin{align*}
 dF(JV)&=d\pi\circ 4J\tau_{JV}\circ d\pi\big|_{\Horiz}^{-1}=-d\pi\circ J\circ 4J\tau_V\circ d\pi\big|_{\Horiz}^{-1}\\
 &=-d\pi\circ J^\Horiz\circ d\pi\big|_\Horiz^{-1}\circ dF(V)=J^-dF(V).\qedhere
\end{align*}
\end{proof}

This completes the proof of Theorem~\ref{nk6reducible}.

\subsection{Further holonomy reduction}

We now turn to the case where the holonomy algebra of the canonical connection of a Gray manifold is properly contained in $\fs(\u(1)\oplus\u(2))$.

\begin{satz}
 Let $(M^6,g,\tau)$ be a geometry with parallel skew torsion which is a Gray manifold. If $\hol(\nabla^\tau)\subsetneq\fs(\u(2)\oplus\u(1))$, then $\hol(\nabla^\tau)=\fs(\u(1)\oplus\u(1)\oplus\u(1))$ and $(M,g,\tau)$ is locally isomorphic to the homogeneous Gray manifold $F_{1,2}$.
\end{satz}
\begin{proof}
 Since $\hol(\nabla^\tau)\subsetneq\fs(\u(2)\oplus\u(1))$, it must be contained in one of the maximal proper subalgebras $\su(2)$ and $\fs(\u(1)\oplus\u(1)\oplus\u(1))$.

 In the first case, the $\C$-summand of the holonomy representation $\C^2\oplus\C$ would have to be trivial, which means that there exists an at least two-dimensional space of $\nabla^\tau$-parallel vector fields. However this is impossible by Lemma~\ref{nk6parallel}.

 In the second case, the holonomy representation splits further into $\C\oplus\C\oplus\C$. It is easy to see that everything is vertical with respect to the canonical $\fs(\u(1)\oplus\u(1)\oplus\u(1))$-splitting, and thus by Lemma~\ref{verthom}, $(M,g,\tau)$ is an Ambrose--Singer manifold.

 The holonomy algebra cannot be smaller than $\h:=\fs(\u(1)\oplus\u(1)\oplus\u(1))$, since otherwise we would again have trivial summands in the holonomy representation and thus parallel vector fields, which Lemma~\ref{nk6parallel} forbids. Hence $\hol(\nabla^\tau)=\h$.

 Consider thus the transvection algebra $\g=\h\oplus\m$, where $\m=\C^3$ is the holonomy representation. As before, we have $\tau\in(\m\otimes\su(3)^\perp)^{\su(3)}\subset(\m\otimes\h^\perp)^{\h}$ and $(M,g,\tau)$ is Einstein with positive scalar curvature, so Lemma~\ref{ascompact} tells us that $\g$ is compact and semisimple. The only possibility in dimension eight is $\g=\su(3)$, and since this has rank two, $\h$ is a maximal torus in $\g$. The corresponding Lie groups are $G=\SU(3)$ and $H=T^2$. In particular $H$ is closed in $G$. Thus $(M,g,\tau)$ is locally isometric to the Gray manifold $F_{1,2}=\SU(3)/T^2$.
\end{proof}

\clearpage

\section{Nearly parallel $\rmG_2$-manifolds with reducible holonomy}

Recall from Example~\ref{sasaki} that a \threead\ manifold is a Riemannian manifold $(M,g)$ carrying three sets of structure tensors $(\xi_i,\Phi_i)_{i=1,2,3}$ such that the algebraic (\emph{almost 3-contact metric}) conditions \eqref{3sas1}, \eqref{3sas2} and \eqref{3adcond1} as well as the differential condition \eqref{3adcond2} are satisfied.

As observed in \cite[Thm.~4.5.1]{3ad}, every 7-dimensional \threead\ manifold carries a natural $\rmG_2$-structure
\[\varphi:=\sum_i\xi_i\wedge\Phi_i^\Horiz+\xi_1\wedge\xi_2\wedge\xi_3\]
which is \emph{cocalibrated}, i.e.~$d^\ast\varphi=0$, and whose characteristic connection (in the sense of \cite{srni}) coincides with the canonical \threead\ connection, whose holonomy algebra is contained in $\so(4)\oplus\so(3)$. The $\rmG_2$-structure $\varphi$ defined above is nearly parallel if and only if $\delta=5\alpha$. In this case, its holonomy algebra is also contained in $\g_2$, hence in a maximal $\so(4)$-subalgebra of $\g_2$ (which is characterized by preserving a splitting $\R^7=\R^4\oplus\R^3$ of the standard representation of $\rmG_2$).

Our goal in this section is to show a sort of converse:

\begin{satz}
\label{ng2reducible}
If $(M^7,g,\tau)$ is a geometry with parallel skew torsion which is strictly nearly parallel $\rmG_2$ with reducible holonomy representation, i.e.~$\hol(\nabla^\tau)\subseteq\so(4)\subset\g_2$, then it is \threeadn, and $\delta=5\alpha$.
\end{satz}

\subsection{Recovering the \threead\ structure}

Let $(M^7,g,\tau)$ be as above. Recall from Example~\ref{nearlyg2} that $\tau=\frac{\tau_0}{12}\varphi$, where $\varphi$ is the $\rmG_2$-structure. We assume without restriction that $g$ is the metric induced by $\varphi$, i.e.~that $|\varphi|^2=7$. As in Proposition~\ref{ng2hol}, let $TM=\Horiz\oplus\Verti$ be the $\so(4)$-canonical splitting, where fiberwise $\dim\Horiz=4$ and $\dim\Verti=3$. Note that $\tau^\Horiz=0$, since any subalgebra of $\so(4)$ that stabilizes an element of $\Lambda^3\Horiz$ also stabilizes a vector in $\Horiz$, and by Lemma~\ref{ng2parallel} the holonomy algebra $\hol(\nabla^\tau)$ cannot stabilize a tangent vector. Thus $\tau=\tau^\mixed+\tau^\Verti$.

Since the $3$-form $\varphi$ describes a vector cross product on $TM$, it satisfies the identities \cite[(2.7), (2.13)]{FG}
\begin{align}
 2\varphi_{\varphi_XY}+[\varphi_X,\varphi_Y]&=3X\wedge Y,\label{g2comm}\\
 \{\varphi_X,\varphi_Y\}&=-2\langle X,Y\rangle\Id+X\odot Y,\label{g2anticomm}
\end{align}
where $\{\cdot,\cdot\}$ is the anticommutator, and $X\odot Y=X\otimes Y+Y\otimes X$. Moreover, since $\varphi^\Verti$ is a 3-dimensional vector cross product on $\Verti$, we additionally have
\begin{equation}
 [\varphi_U,\varphi_V]W=\varphi_{\varphi_UV}W=(U\wedge V)W,\qquad U,V,W\in\Verti.\label{g2commvert}
\end{equation}
In light of \eqref{g2comm}, the curvature identity \eqref{curvHHV} reduces to
\begin{equation}
 R^\tau(X,Y)V=12\tau_{\tau_XY}V,\qquad X,Y\in\Horiz,\ V\in\Verti.\label{curvHHV2}
\end{equation}
In addition to the identities \ref{curvature} that hold for any admissible splitting, we also need to describe the purely vertical part of the curvature in our situation.

\begin{lem}
\label{curvvert}
 For any $U,V,W\in\Verti$, we have $R^\tau(U,V)W=-24\tau_{\tau_UV}W$.
\end{lem}
\begin{proof}
 By Lemma~\ref{verthom}, the vertical part $R^{\tau,\Verti}:=\pr_{\Lambda^2\Verti}\circ R^\tau\circ\pr_{\Lambda^2\Verti}$ is $\so(4)$-invariant, i.e.
 \[R^{\tau,\Verti}\in(\Sym^2\Lambda^2\Verti)^{\so(4)}.\]
 Since $\Lambda^2\Verti\cong\Verti\cong\so(3)$ is irreducible, any invariant symmetric endomorphism must be some multiple of the identity, and since the isomorphism $\Lambda^2\Verti\cong\Verti$ is given by $\tau^\Verti$, we must have
 \[R^\tau(U,V)W=c\cdot\tau_{\tau_UV}W,\qquad U,V,W\in\Verti,\]
 for some $c\in\R$ that it remains to determine.

 Recall that the Ricci endomorphism can be written as
 \begin{equation}
 \Ric^\tau=\sum_{i<j}(e_i\wedge e_j)\circ R^\tau(e_i,e_j).\label{ricendo}
 \end{equation}
 for any orthonormal basis $(e_i)$ of $TM$. Recall also that on an Euclidean vector space $T$ of dimension $n$, the \emph{Casimir operator}
 \[\Cas^{\so(n)}_T:=-\sum_{i<j}(v_i\wedge v_j)^2\in(\End T)^{\so(n)},\]
 where $(v_i)$ is an orthonormal basis of $T$, acts as the operator $(n-1)\Id$. By \eqref{curvHV}, we have $R^\tau(\Horiz,\Verti)=0$, and by \eqref{curvHHV2}, $R^\tau(\Horiz,\Horiz)\Verti\subseteq\Verti$. With $(U_i)$ as an orthonormal basis of $\Verti$, and using \eqref{g2commvert}, the only remaining terms in \eqref{ricendo} applied to $W\in\Verti$ are
 \begin{align*}
  \Ric^\tau(W)&=\sum_{i<j}(U_i\wedge U_j)R^\tau(U_i,U_j)W=c\sum_{i<j}(U_i\wedge U_j)\tau_{\tau_{U_i}U_j}W\\
  &=c\left(\frac{\tau_0}{12}\right)^2\sum_{i<j}(U_i\wedge U_j)^2 W=-c\left(\frac{\tau_0}{12}\right)^2\Cas^{\so(3)}_\Verti W =-2c\left(\frac{\tau_0}{12}\right)^2W
 \end{align*}
 Now, \cite[(5.33)]{AS} combined with the fact that $(M,g)$ is Einstein with $\Ric^g=\frac{3\tau_0^2}{8}\Id$ implies that
 \[\Ric^\tau=\Ric^g+7\left(\frac{\tau_0}{12}\right)^2\Cas^{\so(7)}_{TM}=48\left(\frac{\tau_0}{12}\right)^2\Id.\]
 Comparing the above results, we obtain $c=-24$.
\end{proof}

We now define a new connection that will help us reconstruct the \threead\ structure tensors. Let
\[\nabla:=\nabla^g+\tau^\mixed-5\tau^\Verti=\nabla^\tau-6\tau^\Verti.\]
Since both $\nabla^\tau_X$ and $\tau^\Verti_X$ preserve the splitting $TM=\Horiz\oplus\Verti$ for any $X\in TM$, so does $\nabla_X$. In particular $\nabla$ restricts to a connection on the vector bundle $\Verti$.

\begin{lem}
\label{newconn1}
 The connection $\nabla$ on $\Verti$ is flat.
\end{lem}
\begin{proof}
 Let $X,Y$ be horizontal and $U,V,W$ be vertical vector fields on $M$, and $R$ denote the curvature of $\nabla$. Using $\nabla_X=\nabla^\tau_X$, $\nabla^\tau_XY\in\Horiz$, $\tau_XY\in\Verti$, and \eqref{curvHHV2}, we calculate
 \begin{align*}
  R(X,Y)V&=\nabla_X\nabla_YV-\nabla_Y\nabla_XV-\nabla_{[X,Y]}V\\
  &=\nabla^\tau_X\nabla^\tau_YV-\nabla^\tau_Y\nabla^\tau_XV-\nabla^\tau_{\nabla^\tau_XY-\nabla^\tau_YX}V+2\nabla_{\tau_XY}V\\
  &=R^\tau(X,Y)V-12\tau_{\tau_XY}V=0
 \end{align*}
 Further, with $\nabla^\tau_XV\in\Verti$, $\nabla^\tau_VX\in\Horiz$, $\tau_XV\in\Horiz$, $\nabla^\tau\tau=0$, and \eqref{curvHV}, we obtain
 \begin{align*}
  R(X,V)W&=\nabla_X\nabla_VW-\nabla_V\nabla_XW-\nabla_{[X,V]}W\\
  &=\nabla^\tau_X(\nabla^\tau_VW-6\tau_VW)-(\nabla^\tau_V-6\tau_V)\nabla^\tau_XW-\nabla_{\nabla^\tau_XV-\nabla^\tau_VX-2\tau_XV}W\\
  &=\nabla^\tau_X\nabla^\tau_VW-\nabla^\tau_V\nabla^\tau_XW-6\nabla^\tau_X(\tau_VW)+6\tau_V\nabla^\tau_XW\\
  &\quad-\nabla^\tau_{\nabla^\tau_XV-\nabla^\tau_VX-2\tau_XV}W+6\tau_{\nabla^\tau_XV}W\\
  &=R^\tau(X,V)W-6(\nabla^\tau_X\tau)_VW=0.
 \end{align*}
 Finally, we use $[U,V],\nabla^\tau_VW,\tau_VW\in\Verti$, $\nabla^\tau\tau=0$, \eqref{g2commvert}, and Lemma~\ref{curvvert}, to see that
 \begin{align*}
  R(U,V)W&=\nabla_U\nabla_VW-\nabla_V\nabla_UW-\nabla_{[U,V]}W\\
  &=\nabla_U(\nabla^\tau_VW-6\tau_VW)-\nabla_V(\nabla^\tau_UW-6\tau_UW)-\nabla^\tau_{[U,V]}W+6\tau_{[U,V]}W\\
  &=R^\tau(U,V)W-6\tau_U(\nabla^\tau_VW-6\tau_VW)-6\nabla^\tau_U(\tau_VW)\\
  &\quad+6\tau_V(\nabla^\tau_UW-6\tau_UW)+6\nabla^\tau_V(\tau_UW)+6\tau_{\nabla^\tau_UV-\nabla^\tau_VU-2\tau_UV}W\\
  &=R^\tau(U,V)W-6(\nabla^\tau_U\tau)_VW+6(\nabla^\tau_V\tau)_UW+36[\tau_U,\tau_V]W-12\tau_{\tau_UV}W\\
  &=R^\tau(U,V)W-24\tau_{\tau_UV}W=0.\qedhere
 \end{align*}
\end{proof}

\begin{lem}
\label{newconn2}
 $\nabla\varphi^\Verti=0$.
\end{lem}
\begin{proof}
 Since $\tau=\frac{\tau_0}{12}\varphi$ and $\nabla^\tau\tau^\Verti=0$, it remains to show that $(\varphi^\Verti_X)_\ast\varphi^\Verti=0$. But $\varphi^\Verti\in\Lambda^3\Verti$ and $\dim\Verti=3$, so clearly
 \[b((\varphi^\Verti)^2)\in \Verti^{\otimes 4}\cap\Lambda^4TM=\Lambda^4\Verti=0,\]
 hence $(\varphi^\Verti_X)_\ast\varphi^\Verti=X\intprod b((\varphi^\Verti)^2)=0$.
\end{proof}

\begin{proof}[Proof of Theorem~\ref{ng2reducible}]
 At an arbitrary point of $M$, choose an orthonormal basis $(\xi_1,\xi_2,\xi_3)$ of $\Verti$ that is positively oriented with respect to $\varphi^\Verti$, that is
 \[\varphi^\Verti=\xi_1\wedge\xi_2\wedge\xi_3\]
 at that point. By Lemma~\ref{newconn1}, we may extend the $\xi_i$ to $\nabla$-parallel vector fields on $M$, and by Lemma~\ref{newconn2}, the above relation then holds globally. Since $\varphi=\varphi^\mixed+\varphi^\Verti$, we may write
 \[\varphi=\sum_{i=1}^3\xi_i\wedge\omega_i+\xi_1\wedge\xi_2\wedge\xi_3\]
 for the horizontal $2$-forms $\omega_i:=\varphi^\mixed_{\xi_i}=\varphi_{\xi_i}-\xi_j\wedge\xi_k$. Each $\omega_i$ defines an almost complex structure on $\Horiz$, which is easily checked using the identity \eqref{g2anticomm} with $X=Y=\xi_i$. If we now define
 \[\Phi_i:=\omega_i-\xi_j\wedge\xi_k=\varphi_{\xi_i}-2\xi_j\wedge\xi_k,\]
 then this implies
 \[\Phi_i^2=(\omega_i-\xi_j\wedge\xi_k)^2=\omega_i^2-(\xi_j\wedge\xi_k)^2=-\Id+\xi_i\otimes\xi_i,\]
 verifying \eqref{3adcond1}. Using the definition of $\nabla$, the condition $\nabla\xi_i=0$ is rewritten as
 \[\nabla^g\xi_i=-\xi_i\intprod(\tau^\mixed-5\tau^\Verti)=-\frac{\tau_0}{12}(\omega_i-5\xi_j\wedge\xi_k)\]
 for any even permutation $(i,j,k)$ of $(1,2,3)$. This is turn is equivalent to \eqref{3adcond2} with $\alpha=-\frac{\tau_0}{12}$ and $\delta=5\alpha$. Since the $\omega_i$ are horizontal, we clearly have $\Phi_i\xi_j=-\xi_k=-\Phi_j\xi_i$, which is \eqref{3sas1}. Finally, combining the identities \eqref{g2comm} and \eqref{g2anticomm}, we obtain
 \[\varphi_X\circ\varphi_Y=2X\otimes Y-Y\otimes X-\langle X,Y\rangle\Id-\varphi_{\varphi_XY}\in\End TM.\]
 Specializing to $X=\xi_i$ and $Y=\xi_j$, we find $\varphi_{\xi_i}\circ\varphi_{\xi_j}=2\xi_i\otimes\xi_j-\xi_j\otimes\xi_i-\varphi_{\xi_k}$. Thus
 \begin{align*}
  \Phi_i\Phi_jX&=(\varphi_{\xi_i}-2\xi_j\wedge\xi_k)(\varphi_{\xi_j}-2\xi_k\wedge\xi_i)X\\
  &=\varphi_{\xi_i}\varphi_{\xi_j}X-2\langle\xi_k,X\rangle\varphi_{\xi_i}\xi_i+2\langle\xi_i,X\rangle\varphi_{\xi_i}\xi_k\\
  &\quad-2\langle\xi_j,\varphi_{\xi_j}X\rangle\xi_k+2\langle\xi_k,\varphi_{\xi_j}X\rangle\xi_j+4\langle\xi_i,X\rangle\xi_j\\
  &=\varphi_{\xi_i}\varphi_{\xi_j}X\\
  &=-\varphi_{\xi_k}X+2\langle\xi_i,X\rangle\xi_j-\langle\xi_j,X\rangle\xi_i\\
  &=-\Phi_kX+\langle\xi_j,X\rangle\xi_i
 \end{align*}
 and similarly for $\Phi_j\Phi_iX$, thus showing the desired relation \eqref{3sas2}.
\end{proof}

\subsection{Further holonomy reduction}

\begin{satz}
 If $\hol(\nabla^\tau)\subseteq\so(4)$ inside $\g_2$, then either we have equality, or the base of the locally defined submersion is Kähler--Einstein and $\hol(\nabla^\tau)=\u(2)\subset\so(4)$.
\end{satz}
\begin{proof}
 Since $\Verti$ is $\hol(\nabla^\tau)$-invariant and $\dim\Verti=3$, it must be irreducible -- otherwise, there would be a trivial summand in $\Verti$ and thus a $\nabla^\tau$-parallel vector field, which is impossible by Lemma~\ref{ng2parallel}. Hence the projection $\hol(\nabla^\tau)\to\so(\Verti)$ is surjective.

 Assume that $\hol(\nabla^\tau)\subsetneq\so(4)$, and write $\so(4)\cong\so(3)_1\oplus\so(3)_2$ such that $\Verti\cong\so(3)_1$, and both $\so(3)_{1,2}$ act irreducibly on $\Horiz$. The maximal proper subalgebras of $\so(4)$ are, up to conjugacy,
 \[\so(3)_1\oplus\u(1),\qquad\u(1)\oplus\so(3)_2,\qquad\diag(\so(3))\subset\so(3)_1\oplus\so(3)_2.\]
 Since $\u(1)\oplus\so(3)_2$ has an invariant vector in $\Verti\cong\so(3)_1$ and $\diag(\so(3))$ has an invariant vector in $\Horiz\cong\R^4$, the holonomy algebra cannot be contained in either of them by Lemma~\ref{ng2parallel}. Furthermore, since $\hol(\nabla^\tau)\to\so(\Verti)$ is surjective, the only remaining possibilities are $\hol(\nabla^\tau)=\so(3)_1\oplus\u(1)$ or $\hol(\nabla^\tau)=\so(3)_1$.

 The representation $\Lambda^2\Horiz$ splits under $\so(3)_1\oplus\u(1)$ as
 \[\Lambda^2\Horiz\cong\R\oplus\C\oplus\R^3,\]
 where the $\u(1)$-factor acts nontrivially only on $\C$, and $\so(3)_1$ only on $\R^3$. In particular, there exists an invariant element, corresponding to a $\nabla^\tau$-parallel horizontal $2$-form $\alpha$. This 2-form is nondegenerate -- otherwise its kernel would be a $\nabla^\tau$-parallel subbundle of $\Horiz$, which is impossible since $\Horiz$ is irreducible as a representation of $\so(3)_1$. Thus we may rescale and assume that $\alpha$ is an almost complex structure on $\Horiz$. In turn, $\so(3)_1\oplus\u(1)$ is the $\u(2)$-subalgebra stabilizing $\alpha$.

 The contraction $\alpha\intprod\tau$ is a parallel $1$-form, hence zero. For any vertical vector field $V$, we then calculate using the Cartan formula
 \begin{align*}
 \Lie_V\alpha&=V\intprod d\alpha=\sum_iV\intprod(e_i\wedge\nabla^g_{e_i}\alpha)=-\sum_iV\intprod(e_i\wedge(\tau_{e_i})_\ast\alpha)\\
 &=-(\tau_V)_\ast\alpha+\sum_ie_i\wedge(V\intprod(\tau_{e_i})_\ast\alpha)=-(\tau_V)_\ast\alpha+\sum_ie_i\wedge((\tau_{e_i})_\ast(V\intprod\alpha)-\tau_{e_i}V\intprod\alpha)\\
 &=-(\tau_V)_\ast\alpha+\sum_ie_i\wedge(\tau_Ve_i\intprod\alpha)=-(\tau_V)_\ast\alpha-\sum_i\tau_Ve_i\wedge(e_i\intprod\alpha)=-2(\tau_V)_\ast\alpha.
 \end{align*}
 Since $\so(3)_2$ acts trivially on $\Verti$, the $2$-form $\tau_V$ is also $\so(3)_2$-invariant. Thus, under the identification $\Lambda^2\Horiz\cong\so(4)$, we have $\tau_V\in\so(3)_1$, while $\alpha\in\so(3)_2$. In particular, the skew-symmetric endomorphisms $\tau_V$ and $\alpha$ commute, and \eqref{lam2action} implies
 \[(\tau_V)_\ast\alpha=[\tau_V,\alpha]=0.\]
 Thus $\alpha$ is projectable to a $2$-form $\check\alpha$ on the quaternion-Kähler base $(N,g_N)$ in the canonical $\so(4)$-submersion, and since $\tau^\Horiz=0$, property \ref{base} of the canonical submersion implies that $\nabla^{g_N}\check\alpha=0$. Thus $\check\alpha$ is a $\nabla^{g_N}$-parallel complex structure on $N$, that is, $(N,g_N,\check\alpha)$ is Kähler.

 Finally, we rule out the case $\hol(\nabla^\tau)=\so(3)_1=\su(2)$. Under this subalgebra, the representation $\so(3)_2\subset\Lambda^2\Horiz$ is trivial. By the argument above, we obtain a triple of $\nabla^{g_N}$-parallel complex structures on $N$ satisfying the $\so(3)$ commutation relations. This means that $(N,g_N)$ is hyper-Kähler. However by \cite[Thm.~4.2.10]{stecker} we have
 \[\scal_{g_N}=48\alpha\delta>0,\]
 which clashes with the fact that hyper-Kähler metrics are Ricci-flat.
\end{proof}

\clearpage

\section{Almost irreducible stabilizer actions and Sasaki geometry}

We conclude this article with a particular class of geometries with parallel skew torsion that does not appear in the classification of Theorem~\ref{csextended}, while still being of great importance. In \cite{CMS}, \emph{geometries with torsion of special type} were introduced. These are geometries with parallel skew torsion $(M,g,\tau)$ such that $\hol(\nabla^\tau)$ acts trivially on the vertical space $\Verti\neq0$ of the standard submersion. Equivalently, the distribution $\Verti$ is spanned by $\nabla^\tau$-parallel vector fields. Geometries with torsion of special type include Sasaki manifolds as well as parallel \threead\  manifolds ($\delta=2\alpha$).

\subsection{A characterization of Sasaki manifolds}

Replacing the holonomy action by the stabilizer action, we focus on the so-called \emph{almost irreducible} case, where in the canonical $\stab(\tau)$-splitting $TM=\Horiz\oplus\Verti$, the vertical part $\Verti$ is one-dimensional and $\Horiz$ is an irreducible representation of $\stab(\tau)$.

\begin{satz}\label{sasakisatz}
 If $(M,g,\tau)$ is a geometry with parallel skew torsion such that $\stab(\tau)$ acts almost irreducibly on the tangent space, and $(M,g,\tau)$ has no local one-dimensional factor, then $(M,g,\tau)$ is Sasakian.
\end{satz}
\begin{proof}
 Write the canonical $\stab(\tau)$-splitting as $TM=\Horiz\oplus\Verti$, and let $\xi\in\X(M)$ be a unit length $\nabla^\tau$-parallel vector field spanning $\Verti$. By the irreducibility of $\Horiz$ and the assumption, $(M,g,\tau)$ is indecomposable. Thus Lemma~\ref{specialtype} yields $\tau^\Horiz=0$ and we may write
 \[\tau=\xi\wedge\Phi\]
 for the $\stab(\tau)$-invariant horizontal $2$-form $\Phi:=\tau_\xi$. By Schur's Lemma, its square $\Phi^2\in(\Sym\Horiz)^{\stab(\tau)}$ is a multiple of the identity, i.e.
 \[\Phi^2=-\lambda\Id_\Horiz\]
 for some $\lambda>0$. Up to a joint rescaling of the metric and $\xi$, we can assume that $\lambda=1$. Since $\xi$ and $\Phi$ are $\nabla^\tau$-parallel, we calculate for $X\in\Horiz$
 \begin{align*}
  \nabla^g_X\xi&=-\tau_X\xi=X\intprod\Phi,&\nabla^g_\xi\xi&=-\tau_\xi\xi=0,\\
  \nabla^g_X\Phi&=-\tau_X\Phi=[\xi\wedge\Phi(X),\Phi]=-\lambda X\wedge\xi,&\nabla^g_\xi\Phi&=-\tau_\xi\Phi=-[\Phi,\Phi]=0.
 \end{align*}
 Thus $(\xi,\Phi)$ satisfy the Sasaki conditions \eqref{sasaki}.
\end{proof}

In particular we can view $\Horiz$ as a complex representation of $\stab(\tau)$, using the invariant complex structure $\Phi$. $M^{2n+1}$ is necessarily odd-dimensional, and $\stab(\tau)\cong\u(n)$.

Applying the canonical $\u(n)$-submersion, we recover the well-known fact that any Sasaki manifold locally fibers over a Kähler manifold. Indeed, the torsion of the base vanishes as a consequence of \ref{base} and Lemma~\ref{specialtype}.

\subsection{Holonomy reductions for Sasaki manifolds}

Let $(M^{2n+1},g,\xi,\Phi)$ be a Sasaki manifold, $\tau=\xi\wedge\Phi$ as in~\ref{sasakibsp}, and $(N^{2n},g_N,J)$ the Kähler base of the canonical $\u(n)$-submersion $\pi: M\to N$. Again one may raise the question whether it is possible to characterize the cases where $\hol(\nabla^\tau)$ is a proper subalgebra of $\u(n)$. We shall henceforth assume that $n\geq2$, because the three-dimensional case is covered in Proposition~\ref{3dim1} and Remark~\ref{3dim2}.

We consider the (complex) representation $\Horiz$ of $\hol(\nabla^\tau)$ and distinguish three possibilities:
\begin{enumerate}[\upshape(H1)]
 \item $\Horiz=\bigoplus_\alpha\Horiz_\alpha$ is reducible as a complex representation. Then accordingly, $(N,g_N,J)$ locally splits into a product of Kähler manifolds $(N_\alpha,g_\alpha,J_\alpha)$.
 \item $\Horiz$ is irreducible as a complex representation, but reducible as a real one. That is, $\Horiz=\Horiz'\oplus\Phi\Horiz'$ for an irreducible real representation $\Horiz'$ of $\hol(\nabla^\tau)$.
 \label{Sasaki2}
 \item $\Horiz$ is irreducible as a real representation. Then $\hol(\nabla^\tau)\subsetneq\u(n)$ is necessarily a maximal subalgebra.
 \label{Sasaki3}
\end{enumerate}

In the first case, we can only say that the local Kähler base is reducible. The last two cases merit a further investigation. First, however, we need to relate the holonomy of $\nabla^\tau$ to the Riemannian holonomy of the Kähler base. View both holonomy groups as subgroups of $\U(n)$, and denote with $\U(1)\subset\U(n)$ the rotation subgroup generated by $\Phi$ on $\Horiz$, resp.~by the $\pi$-related endomorphism $J=d\pi\circ\Phi\circ d\pi\big|_{\Horiz}^{-1}$ on $TN$.

\begin{satz}
\label{sasakihol}
 With the notation from the beginning of the subsection, we have the inclusions
 \begin{align*}
 \Hol(\nabla^\tau)&\subseteq\Hol(\nabla^{g_N})\cdot\U(1),\\
 \Hol(\nabla^{g_N})&\subseteq\Hol(\nabla^\tau)\cdot\U(1),
 \end{align*}
 and $\U(1)$ is contained in at least one of the holonomy groups.
\end{satz}
\begin{proof}
 Let $\gamma: [0,1]\to N$ be a closed smooth curve, and let $X$ be a vector field along $\gamma$. Let further $\delta: [0,1]\to M$ be a closed smooth curve with $\pi\circ\delta=\gamma$, and denote with $\tilde X$ the horizontal lift of $X$ along $\gamma$. At every point of the curve $\delta$, we may write the tangent vector as
 \[\dot\delta=\tilde{\dot\gamma}+f\xi\]
 for some function $f: [0,1]\to\R$. By \ref{base}, we have
 \[\nabla^\tau_{\tilde{\dot\gamma}}\tilde X=\widetilde{\nabla^{g_N}_{\dot\gamma}X},\]
 and since $\tilde X$ is projectable, $\Lie_\xi\tilde X=0$ and thus
 \[\nabla^g_\xi\tilde X=\nabla^g_{\tilde X}\xi=\Phi\tilde X=\widetilde{JX}.\]
 Putting this together, we find
 \begin{align*}
  \nabla^\tau_{\dot\delta}\tilde X&=\nabla^\tau_{\tilde{\dot\gamma}}\tilde X+f\nabla^\tau_\xi\tilde X=\widetilde{\nabla^{g_N}_{\dot\gamma}X}+f(\nabla^g_\xi\tilde X+\tau_\xi\tilde X)=\widetilde{\nabla^{g_N}_{\dot\gamma}X}+2f\widetilde{JX}.
 \end{align*}
 Thus the parallel transport equation of $\tilde X$ with respect to $\nabla^\tau$ is given by
 \[\nabla^{g_N}_{\dot\gamma}X+2fJX=0.\]
 If $X_N$ is the solution of $\nabla^{g_N}_{\dot\gamma}X=0$ with initial value $X_N(0)=X(0)$, then using $\nabla^{g_N}J=0$ and integrating yields
 \[X(1)\in\spann\{X_N(1),JX_N(1)\}.\]
 Since $\nabla^\tau$ is metric, parallel transport preserves length, and we conclude
 \[\Hol(\nabla^\tau)\subseteq\{\alpha h+\beta h\Phi\,|\,h\in\Hol(\nabla^{g_N}),\ \alpha^2+\beta^2=1\}=\Hol(\nabla^{g_N})\cdot\U(1).\]
 Since also $X_N(1)\in\spann\{X(1),JX(1)\}$, we have at the same time
 \[\Hol(\nabla^{g_N})\subseteq\{\alpha h+\beta hJ\,|\,h\in\Hol(\nabla^{g_N}),\ \alpha^2+\beta^2=1\}=\Hol(\nabla^\tau)\cdot\U(1).\]

 Let us now compare the curvature of $\nabla^\tau$ and $\nabla^{g_N}$. For vector fields $X,Y,Z\in\X(N)$ with horizontal lifts $\tilde X,\tilde Y,\tilde Z\in\X(M)$, \ref{base} implies that
 \begin{align*}
  R^\tau(\tilde X,\tilde Y)\tilde Z&=\widesttilde{R^{g_N}(X,Y)Z}+\nabla^\tau_{\widetilde{[X,Y]}-[\tilde X,\tilde Y]}\tilde Z.
 \end{align*}
 For the difference term, we note that
 \begin{align*}
  \widetilde{[X,Y]}-[\tilde X,\tilde Y]&=\widesttilde{\nabla^{g_N}_XY-\nabla^{g_N}_YX}-\nabla^g_{\tilde X}\tilde Y+\nabla^g_{\tilde Y}\tilde X\\
  &=\nabla^\tau_{\tilde X}\tilde Y-\nabla^\tau_{\tilde Y}\tilde X-\nabla^g_{\tilde X}\tilde Y+\nabla^g_{\tilde Y}\tilde X\\
  &=2\tau_{\tilde X}\tilde Y=2g(\Phi\tilde X,\tilde Y)\xi=2\omega(X,Y)\xi,
 \end{align*}
 where $\omega(X,Y)=g_N(JX,Y)$ is the Kähler form on $N$. Again, since $\Lie_\xi\tilde Z=0$, we have
 \[\nabla^\tau_\xi\tilde Z=\nabla^g_\xi\tilde Z+\tau_\xi\tilde Z=2\Phi\tilde Z=2\widetilde{JZ}.\]
 Finally, we obtain
 \begin{equation}
  R^\tau(\tilde X,\tilde Y)\tilde Z=\widesttilde{R^{g_N}(X,Y)Z}+4\omega(X,Y)\widetilde{JZ}.\label{curvsasaki}
 \end{equation}
 Thus $R^\tau$ is $\pi$-related to $R^{g_N}+4\omega\otimes\omega$. In particular, the curvature operators cannot both annihilate $\omega$ (resp.~$\Phi$). Thus at least one of the holonomy algebras $\hol(\nabla^\tau)$ and $\hol(\nabla^{g_N})$ has to contain $\u(1)=\R\omega$.
\end{proof}

Finally, we are ready to discuss the cases \ref{Sasaki2} and \ref{Sasaki3}. It turns out that it is possible to completely describe \ref{Sasaki2}, and we give a classification scheme for \ref{Sasaki3}.

\begin{satz}
\label{sasakirealred}
 If $n\geq2$, condition \ref{Sasaki2} is satisfied if and only if $(M^{2n+1},g)$ is locally isometric to one of the Sasaki manifolds $\SO(n+2)/\SO(n)$ or $\SO(n,2)/\SO(n)$, which are circle bundles over Hermitian symmetric spaces.
\end{satz}
\begin{proof}
 We have $\Horiz=\Horiz'\oplus\Phi\Horiz'$ as a representation of $\hol(\nabla^\tau)$. Since $\Horiz'\cong\Phi\Horiz'$, everything in $TM\cong\R\oplus\Horiz$ is vertical with respect to the canonical $\hol(\nabla^\tau)$-splitting. By Lemma~\ref{verthom}, $\nabla^\tau R^\tau=0$ and $(M,g,\tau)$ is a naturally reductive Ambrose--Singer manifold. It then follows from \eqref{curvsasaki} and \ref{base} that also $\nabla^{g_N}R^{g_N}=0$. Hence $(N,g_N)$ is locally symmetric, i.e.~locally of the form $G/(K\cdot\U(1))$. Let $\p\cong\Horiz$ be the isotropy representation of $\hol(\nabla^{g_N})=\k\oplus\u(1)$, and let $\g':=\k\oplus\m$ be the transvection algebra of $(M,g,\tau)$, where $\m=\Verti\oplus\Horiz$. It remains to show that $\g'\cong\g$. We stipulate the isomorphism
 \[\psi:\quad \g=\k\oplus\u(1)\oplus\p\longrightarrow\g'=\k\oplus\Verti\oplus\Horiz:\quad (k,J,X)\mapsto(k,-\tfrac12\xi,\tilde X).\]
 The $[\k,\k]$- and $[\k,\p]$-part of the bracket agree by construction. We also have $[\k,\Verti]=0$ since $\nabla^\tau\xi=0$, which agrees with the fact that $[\k,\u(1)]=0$ in $\g$. For the $[\m,\m]$-part, recall the definition \eqref{transvec} of the bracket. Clearly, $[\Verti,\Verti]=0=[\u(1),\u(1)]$. Next, for $[\Horiz,\Horiz]_\m$ we have
 \[[X,Y]_{\g'}=-2\tau_XY=-2\Phi(X,Y)\xi,\qquad X,Y\in\Horiz,\]
 while $[\p,\p]_\p=0$, and it follows from \eqref{curvsasaki} that $[\p,\p]_{\u(1)}$ is given by
 \[([X,Y]_\g)_{\u(1)}=-R^{g_N}(X,Y)_{\u(1)}=4\omega(X,Y)J,\qquad X,Y\in\p,\]
 so indeed $[\Horiz,\Horiz]_\m$ and $[\p,\p]_{\p\oplus\u(1)}$ agree under $\psi$. Next, $[\p,\p]_\k$ and $[\Horiz,\Horiz]_\k$ are given by $-\pr_\k\circ R^{g_N}$ and $-R^\tau\big|_{\Horiz\times\Horiz}$, respectively, which fit together under $\psi$ thanks to \eqref{curvsasaki}. The $[\u(1),\p]$-part is simply given by $[J,X]_\g=JX\in\p$ for $X\in\p$. Meanwhile for $[\Verti,\Horiz]$ we have
 \begin{align*}
  ([\xi,X]_{\g'})_\m&=-2\tau_\xi X=-2\Phi X,\\
  ([\xi,X]_{\g'})_\k&=-R^\tau(\xi,X),
 \end{align*}
 for $X\in\Horiz$, and this agrees with $[\u(1),\p]$ under $\psi$ provided $R^\tau(\xi,\cdot)=0$.

 Since $\nabla^\tau=0$, we have $R^\tau(\xi,\cdot)\xi=0$. Thus it suffices to compute $R^\tau(\xi,\tilde X)\tilde Y$ for $X,Y\in\X(N)$. Using earlier identities, we see that
 \begin{align*}
  R^\tau(\xi,\tilde X)\tilde Y&=\nabla^\tau_\xi\nabla^\tau_{\tilde X}\tilde Y-\nabla^\tau_{\tilde X}\nabla^\tau_\xi\tilde Y-\nabla^\tau_{[\tilde X,\xi]}\tilde Y\\
  &=\nabla^\tau_\xi\widetilde{\nabla^{g_N}_XY}-\nabla^\tau_{\tilde X}(2\widetilde{JY})\\
  &=2\widesttilde{J\nabla^{g_N}_XY-\nabla^{g_N}_X(JY)}=0
 \end{align*}
 since $\nabla^{g_N}J=0$. This finishes the proof that $\g\cong\g'$, and it follows that $(M,g,\tau)$ is locally isometric to $G/K$.

 After inspecting the known list of Hermitian symmetric spaces, we find that the only possibility where $\p$ splits under restriction to $\k$ is the case where $G=\SO(n+2)$ or $\SO(n,2)$, $K=\SO(n)$ and $\p=\R^n\oplus\R^n$. Indeed, for the two naturally reductive spaces $(G/K,g,\tau)$, \eqref{holhom} implies that $\hol(\nabla^\tau)=\k$, so \ref{Sasaki2} is indeed satisfied for these spaces.
\end{proof}

\begin{satz}
\label{sasakirealirred}
 Under the assumptions of \ref{Sasaki3}, if $n\geq2$, one of the following holds:
 \begin{enumerate}[\upshape(a)]
 \item $(M^{2n+1},g)$ is locally isometric to one of the Sasaki manifolds
 \begin{align*}
  \SU(p+q)&/(\SU(p)\times\SU(q)),&\SO(2k)&/\SU(k),&\Sp(k)&/\SU(k),\\
  &&\rmE_6&/\Spin(10),&\rmE_7&/\rmE_6,\\
  \SU(p,q)&/(\SU(p)\times\SU(q)),&\SO(k,\mathbb{H})&/\SU(k),&\Sp(k,\R)&/\SU(k),\\
  &&\rmE_6^{-14}&/\Spin(10),&\rmE_7^{-25}&/\rmE_6,
 \end{align*}
 which are circle bundles over Hermitian symmetric spaces.
 \item $\hol(\nabla^\tau)=\su(n)$ and $(M^{2n+1},g)$ locally fibers over a Kähler--Einstein manifold with scalar curvature $8n^2$.
 \item $\hol(\nabla^\tau)=\sp(n/2)\oplus\u(1)$ and $(M^{2n+1},g)$ locally fibers over a hyper-Kähler manifold.
\end{enumerate}
\end{satz}
\begin{proof}
 Suppose that $\Horiz$ is irreducible under $\hol(\nabla^\tau)$. First, we show that it is also irreducible under $\hol(\nabla^{g_N})$. Indeed, by Theorem~\ref{sasakihol}, $\hol(\nabla^\tau)\subseteq\hol(\nabla^{g_N})+\u(1)$, so if $\Horiz$ was reducible under $\hol(\nabla^{g_N})$, we would have $\Horiz=\Horiz'\oplus J\Horiz'$, where $\Horiz'$ is irreducible under $\hol(\nabla^{g_N})$. But this is impossible, since as a consequence of the de Rham decomposition theorem, Riemannian holonomy representations are multiplicity-free.

 Now, by the Berger classification of Riemannian holonomy groups, $\hol(\nabla^{g_N})$ can be either $\u(n)$, $\su(n)$, $\sp(n/2)$, or $(N,g_N)$ is locally symmetric.

 In the symmetric case, let $(N,g_N)$ be locally isometric to $G/(K\cdot\U(1))$. We may argue as in the proof of Theorem~\ref{sasakirealred} that $(M,g,\tau)$ is locally the naturally reductive space $G/K$, where $(G,K)$ is such that the isotropy representation $\p\cong\Horiz$ of $\k\oplus\u(1)$ does not split when restricted to $\k$. The possibilities are precisely the following:
 \begin{center}
  \renewcommand*{\arraystretch}{1.2}
  \begin{tabular}{c|c|c}\hline
   $G$&$K$&$\p$\\\hline\hline
   $\SU(p+q)$ or $\SU(p,q)$&$\SU(p)\times\SU(q)$&$\C^p\otimes_\C\C^q$\\\hline
   $\SO(2k)$ or $\SO(k,\mathbb{H})$&$\SU(k)$&$\Lambda^2\C^k$\\\hline
   $\Sp(k)$ or $\Sp(k,\R)$&$\SU(k)$&$\Sym^2\C^k$\\\hline
   $\rmE_6$ or $\rmE_6^{-14}$&$\Spin(10)$&$\C^{16}=(\mathbb{O}\otimes\C)^2$\\\hline
   $\rmE_7$ or $\rmE_7^{-25}$&$\rmE_6$&$\C^{27}=\mathrm{Herm}(3,\mathbb{O})\otimes\C$\\\hline
  \end{tabular}
 \end{center}
 Moreoever, for each of these spaces $G/K$, we indeed have $\hol(\nabla^\tau)=\k$ by \eqref{holhom} and thus \ref{Sasaki3} is satisfied.

 Let us now turn to the situation where $(N,g_N)$ is not locally symmetric. Since by assumption, $\hol(\nabla^\tau)\subsetneq\u(n)$, but at least one of $\hol(\nabla^\tau)$ and $\hol(\nabla^{g_N})$ contains $\u(1)$ by Theorem~\ref{sasakihol}, the case $\hol(\nabla^{g_N})=\su(n)$ is ruled out.

 Assume that $\hol(\nabla^{g_N})=\u(n)$. Then it follows from the $\hol(\nabla^\tau)\subsetneq\u(n)$ and Theorem~\ref{sasakihol} that $\hol(\nabla^\tau)=\su(n)$. Then $R^\tau\in\Sym^2\su(n)$, which implies that $R^\tau(\Phi)=0$. Together with the fact that the Ricci tensor of a Kähler manifold satisfies
 \[\Ric^{g_N}(X,Y)=\frac{1}{2}\tr(R^{g_N}(X,JY)\circ J),\]
 it follows from \eqref{curvsasaki} that
 \begin{align*}
  \Ric^{g_N}(X,Y)&=\frac12\sum_ig_N(R^{g_N}(X,JY)Je_i,e_i)=-2\sum_i\omega(X,JY)g_N(J^2e_i,e_i)=4ng_N(X,Y)
 \end{align*}
 where $(e_i)$ is an orthonormal basis of $TN$. Thus $(N^{2n},g_N)$ is Einstein with scalar curvature $8n^2$.

 Finally, if we assume that $\hol(\nabla^{g_N})=\sp(n/2)$, then $(N,g_N)$ is hyper-Kähler, and it follows from Theorem~\ref{sasakihol} that $\hol(\nabla^\tau)=\sp(n/2)\oplus\u(1)$.
\end{proof}

\section*{Acknowledgments}

The first named author was partially supported by the PNRR-III-C9-2023-I8 grant CF 149/31.07.2023 {\em Conformal Aspects of Geometry and Dynamics}. The second named author acknowledges support by the BRIDGES project funded by ANR grant no.~\textbf{ANR-21-CE40-0017} and by the Procope project no.~\textbf{48959TL}. 

We are grateful to Ilka Agricola, Fabio Podestà, Leander Stecker and Andrew Swann for their helpful comments.


\begin{thebibliography}{uf}
\bibitem{srni} I.~Agricola: \emph{The Srní Lectures on non-integrable geometries with torsion}, Arch. Math. (Brno)~\textbf{42}, pp.~5--84, 2006.
\bibitem{3ad} I.~Agricola, G.~Dileo: \emph{Generalizations of 3-Sasakian manifolds and skew torsion}, Adv. Geom.~\textbf{20}, no.~3, pp.~331--374, 2020.
\bibitem{3adsubm} I.~Agricola, G.~Dileo, L.~Stecker: \emph{Homogeneous non‑degenerate 3‑$(\alpha,\delta)$‑Sasaki manifolds and submersions over quaternionic Kähler spaces}, Ann. Glob. Anal. Geom.~\textbf{60}, pp.~111--141, 2021.
\bibitem{alexandrovhermitian} B.~Alexandrov: \emph{$\Sp(n)\U(1)$-connections with parallel totally skew-symmetric torsion}, J. Geom. Phys.~\textbf{57}, pp.~323--337, 2006.
\bibitem{AS} B.~Alexandrov, U.~Semmelmann: \emph{Deformations of nearly parallel $G_2$ structures}, Asian J. Math.~\textbf{16}, pp.~713--744, 2012.
\bibitem{besse} A.~L.~Besse: \emph{Einstein manifolds}, Classics in Mathematics, Springer, 1987.
\bibitem{butruille} J.-B.~Butruille: \emph{Classification des variétés approximativement kähleriennes homogènes}, Ann. Glob. Anal. Geom.~\textbf{27}, pp.~201--225, 2005.
\bibitem{BM} F.~Belgun, A.~Moroianu: \emph{Nearly Kähler 6-Manifolds with Reduced Holonomy}, Ann. Glob. Anal. Geom.~\textbf{19}, pp.~307--319, 2001.
\bibitem{berger} M.~Berger: \emph{Sur les groupes d'holonomie homogènes de variétés à connexion affine et des variétés riemanniennes}, Bull. Soc. Math. France \textbf{83}, pp.~279--330, 1955.
\bibitem{CCL} G.~Calvaruso, M.~Castrillón-López: \emph{Pseudo-Riemannian Homogeneous Structures}, Developments in Mathematics~\textbf{59}, Springer, 2019.
\bibitem{cleyton} R.~Cleyton: \emph{$G$-structures and Einstein Metrics}, PhD thesis, University of Southern Denmark, 2001.
\bibitem{CS} R.~Cleyton, A.~Swann: \emph{Einstein metrics via intrinsic or parallel torsion}, Math.~Z.~\textbf{247}, pp.~513--528, 2004.
\bibitem{CMS} R.~Cleyton, A.~Moroianu, U.~Semmelmann: \emph{Metric connections with parallel skew-symmetric torsion}, Adv.~Math.~\textbf{307}, no.~107519, 2021.
\bibitem{DL} G.~Dileo, A.~Lotta: \emph{A note on Riemannian connections with skew torsion and the de Rham splitting}, manuscripta math.~\textbf{156}, pp.~299--302, 2018.
\bibitem{FG} M.~Fernandez, A.~Gray: \emph{Riemannian manifolds with structure group $G_2$}, Ann. Mat. Pura Appl.~\textbf{132} (4), pp.~19--45, 1982.
\bibitem{g2paralleltorsion} T.~Friedrich: \emph{$\rmG_2$-manifolds with parallel characteristic torsion}, Diff. Geom. Appl.~\textbf{25}, pp.~632--648, 2007.
\bibitem{nearlyg2} T.~Friedrich, I.~Kath, A.~Moroianu, U.~Semmelmann: \emph{On nearly parallel $\rmG_2$-structures}, J. Geom. Phys.~\textbf{23}, pp.~259--286, 1997.
\bibitem{FK} T.~Friedrich, H.~Kurke: \emph{Compact four-dimensional self-dual Einstein manifolds with positive scalar curvature}, Math. Nachr.~\textbf{106}, pp.~271--299, 1982.
\bibitem{FH} W.~Fulton, J.~Harris: \emph{Representation Theory. A First Course}, Springer, 2004.
\bibitem{gray} A.~Gray: \emph{The Structure of Nearly Kähler manifolds}, Math. Ann.~\textbf{223}, pp.~233--248, 1976.
\bibitem{MS} A.~Moroianu, U.~Semmelmann: \emph{Infinitesimal Einstein deformations of nearly Kähler metrics}, Trans. Amer. Math. Soc.~\textbf{363}, pp.~3057--3069, 2011.
\bibitem{muskarov} O.~Muskarov: \emph{Structures presque hermitiennes sur des espaces twistoriels et leur types}, C.R. Acad. Sci. Paris~\textbf{305}, pp.~307--309, 1987.
\bibitem{nagy} P.-A.~Nagy: \emph{Nearly Kähler geometry and Riemannian foliations}, Asian J. Math.~\textbf{6} (3), pp.~481--504, 2002.
\bibitem{skewhol} C.~Olmos, S.~Reggiani: \emph{The skew-torsion holonomy theorem and naturally reductive spaces}, J. Reine Angew. Math.~\textbf{664}, pp.~29--53, 2012.
\bibitem{reyes} R.~Reyes-Carrión: \emph{Some special geometries defined by Lie groups}, PhD Thesis, Oxford, 1993.
\bibitem{schwachh} L.~J.~Schwachhöfer: \emph{Connections with irreducible holonomy representations}, Adv. Math.~\textbf{160} (1), pp. 1--80, 2001.
\bibitem{stecker} L.~Stecker: \emph{On 3-$(\alpha,\delta)$-Sasaki manifolds and their canonical submersions}, PhD thesis, Universität Marburg, 2021.
\bibitem{steckerNK} L.~Stecker: \emph{Canonical Submersions in Nearly Kähler Geometry} (pre\-print), \href{https://arxiv.org/abs/2211.14012}{arXiv:2211.14012}, 2022.
\bibitem{storm} R.~Storm: \emph{Structure theory of naturally reductive spaces}, Diff. Geom. Appl.~\textbf{64}, pp.~174--200, 2019.
\bibitem{tricerri} F.~Tricerri: \emph{Locally homogeneous Riemannian manifolds}, Rend. Sem. Mat. Univ. Politec. Torino~\textbf{50}, no.~4, pp.~411--426, 1993.
\bibitem{wolf} J.~A.~Wolf: \emph{The Geometry and Structure of Isotropy Irreducible Homogeneous
Spaces}, Acta Math.~\textbf{120}, pp.~59--148, 1968.
\end{thebibliography}
\end{document}